\documentclass[10pt,a4paper]{article}
\usepackage[utf8]{inputenc}
\usepackage{tikz}
\usepackage[width=14.00cm, height=25.00cm]{geometry}
\usepackage[T1]{fontenc}
\usepackage{amsmath,amssymb,amsfonts,amsthm,amscd, bm, mathtools}
\usepackage{enumerate}
\usepackage{todonotes}
\usepackage{tikz}
\usepackage{color}
\usepackage{hyperref}
\usetikzlibrary{positioning}
\usetikzlibrary{decorations.pathmorphing}
\usetikzlibrary{decorations.text}
\newtheorem{thm}{Theorem}[section]
\newtheorem*{thm5}{Theorem \ref{thm5}}
\newtheorem*{thm2}{Theorem \ref{thm2}}
\newtheorem*{thm3}{Theorem \ref{thm3}}
\newtheorem*{thm4}{Theorem \ref{thm4}}
\newtheorem*{cor1}{Corollary \ref{cor1}}
\newtheorem*{thm6}{Corollary \ref{thm6}}
\newtheorem{prop}[thm]{Proposition}
\newtheorem{lemma}[thm]{Lemma}
\newtheorem{cor}[thm]{Corollary}
\newtheorem{defn}[thm]{Definition}

\newtheorem{conj}[thm]{Conjecture}

\newtheorem{claim}{Claim}

\title{On some perfect matching conjectures in infinite, cubic, bridgeless graphs}
\author{Paulo Magalhães Júnior and Antonio Kelson Silva}
\date{}

\newcommand{\Addresses}{{
  \bigskip
  \footnotesize
    P.~Magalhães Jr, \textsc{Instituto Federal do Rio Grande do Norte\\
	Rua Manoel Lopes Filho, 773, Currais Novos - RN, 59380-000, Brazil}\par\nopagebreak
  \textit{E-mail address}, P.~Magalhães Jr: \texttt{paulo.magalhaes@ifrn.edu.br}
  
  \medskip
  A.~Silva, \textsc{Departamento de Matemática, Universidade Federal do Piauí\\
  Rua Dirce Oliveira, 1521-1655 - Ininga, Teresina - PI, 64048-550, Brazil}\par\nopagebreak
  \textit{E-mail address}, A.~Silva: \texttt{kelson.mat@ufpi.edu.br}
 
}}
\begin{document}

\maketitle

\begin{abstract}
The Berge-Fulkerson Conjecture states that every bridgeless cubic graph has six perfect matchings such that each edge belongs to exactly two of them. This conjecture has remained open since 1971, and several of its weakenings have been proposed over the years. Two of the most prominent are the Fan-Raspaud Conjecture and the M\'a\v{c}ajov\'a-\v{S}koviera Conjecture. It is well known that the Berge-Fulkerson Conjecture implies the Fan-Raspaud Conjecture, which in turn implies the M\'a\v{c}ajov\'a-\v{S}koviera Conjecture. These problems have been studied for years in the context of finite graphs, and many equivalences between them and other results have been established. However, little to nothing is known about them in the context of infinite graphs. In this paper, we investigate whether these conjectures remain valid in the infinite setting, establish their implications in analogy to the finite case, and prove that their finite versions are equivalent to their respective infinite versions.
\end{abstract}

\section{Introduction}
\paragraph{}

Perfect matchings in bridgeless cubic graphs have been extensively studied for decades within graph theory. One of the most important conjectures on this topic, which has been an object of study for more than half a century, is the Berge-Fulkerson Conjecture, introduced in \cite{Fulkerson}.

\begin{conj}[Berge-Fulkerson \cite{Fulkerson}] \label{conjecture1}
Every bridgeless cubic graph $G$ has six perfect matchings such that each edge of $G$ is covered by exactly two of them.    
\end{conj}

Over the years, numerous works have been published aimed at studying the Berge-Fulkerson Conjecture and its weakenings, such as \cite{GOEDGEBEUR2026104366, BergeFulkersoncoloringfor-linkedgraphs, BergesConjectureforCubicGraphsWithSmallColouringDefect, ReductionoftheBerge-Fulkersonconjecture, Anequivalentformulationofthe}. With the goal of developing strategies to settle the Berge-Fulkerson Conjecture, several of its weakenings have been proposed. However, little to nothing is known about the Berge-Fulkerson Conjecture or its weakenings in the context of infinite graphs. In this paper, we are interested in studying how this conjecture and its weakenings behave in infinite graphs and how they relate to their finite counterparts. One conjecture that stands out as one of the most prominent and well-known weakenings is the Fan-Raspaud Conjecture, introduced in \cite{Raspaud}.  

\begin{conj}[Fan-Raspaud \cite{Raspaud}]\label{conjecture2}
Every bridgeless cubic graph $G$ contains three perfect matchings $M_1 ,M_2, M_3$ such that no edge is covered by all of them.
\end{conj}

It is immediate to see that Conjecture \ref{conjecture1} implies Conjecture \ref{conjecture2}, regardless of the cardinality of the graph. Indeed, if $G$ is a bridgeless cubic graph containing a collection of 6 perfect matchings where each edge of $G$ belongs to exactly two of them, it suffices to select any subcollection consisting of 3 of these matchings to ensure that their intersection is empty. Although the Fan-Raspaud Conjecture is a weakening of the Berge-Fulkerson Conjecture, it still proves to be remarkably elusive; more than 30 years after its formulation, a proof has yet to be obtained. Consequently, new weakenings have been introduced to bring forth new results and techniques within this theory. A well-known example is the M\'a\v{c}ajov\'a and \v{S}koviera Conjecture, presented in \cite{MACAJOVA2005112}.

\begin{conj}[M\'a\v{c}ajov\'a and \v{S}koviera \cite{MACAJOVA2005112}]\label{conjecture3}
Every bridgeless cubic graph has two perfect matchings $M_1, M_2$ so that $M_1\cap M_2 $ does not contain an odd edge-cut.
\end{conj}

It was shown in \cite{perfectmatchingswithrestricted} and \cite{MACAJOVA2005112} that the Fan-Raspaud Conjecture implies the M\'a\v{c}ajov\'a and \v{S}koviera Conjecture for finite graphs. However, classical parity arguments on joins fail when a cut separates two infinite components; furthermore, approaches based on network flow theory face difficulties due to flows escaping to infinity or failing to behave well. Thus, the validity of this implication in the infinite case is not a direct consequence of the finite arguments. To address this, we propose an alternative approach via our Theorem~\ref{thm4} and Corollary~\ref{cor1}. Although the M\'a\v{c}ajov\'a and \v{S}koviera Conjecture is a weakening of Conjectures \ref{conjecture1} and \ref{conjecture2}, it nevertheless remains wide open. Accordingly, Mazzuoccolo introduced in \cite{Newconjecturesonperfectmatchingsincubicgraphs} the following new weakening of the aforementioned conjectures.

\begin{conj}[Mazzuoccolo \cite{Newconjecturesonperfectmatchingsincubicgraphs}]\label{conjecture4}
Every bridgeless cubic graph has two perfect matchings $M_1, M_2$ so that $G \setminus (M_1\cup M_2)$ is bipartite.
\end{conj}

Recently, Kardo\v{s}, M\'a\v{c}ajov\'a, and Zerafa proved in \cite{Disjointoddcircuitsinabridgeless} that Conjecture \ref{conjecture4} is true. In this paper, we establish the connections between the finite and infinite versions of all the presented conjectures. Our results complete the diagram illustrated in Figure~\ref{fig5}. To achieve this, we employ the concept of $F$-limit, which was introduced by Aurichi, Magalhães Júnior, and Seixas in \cite{F-limit}. Intuitively, an $F$-limit $H$ on a graph $G$ is a subgraph such that there exists a sequence of subgraphs $\langle H_n\rangle_{n\in\mathbb{N}}$ of $G$ where a vertex or edge of $G$ is in $H$ if there are many $n\in\mathbb{N}$ to which the vertex or edge belongs to $H_n$. As usual, we will use an ultrafilter to be precise about what we mean by many. This concept will be formally detailed in Section \ref{section2}. Although it has been only recently formalized, it already finds solid applications in several decomposition and covering problems in infinite graphs, such as the Cycle Faithful Covering Conjecture \cite{F-limit}, Nash-Williams' theorem for edge-connectivity preserving orientations (which remains a conjecture in the infinite case) \cite{aurichi2025orientationspreservingedgeconnectivityinfinite}, and Salia's Conjecture \cite{saliainfinite}.

\begin{figure}[ht]
    \centering

\tikzset{every picture/.style={line width=0.75pt}} 

\begin{tikzpicture}[x=0.65pt,y=0.65pt,yscale=-1,xscale=1]

\draw    (158.33,40.33) -- (158.33,73.33) -- (158.33,87.33) ;
\draw [shift={(158.33,90.33)}, rotate = 270] [fill={rgb, 255:red, 0; green, 0; blue, 0 }  ][line width=0.08]  [draw opacity=0] (8.93,-4.29) -- (0,0) -- (8.93,4.29) -- cycle    ;
\draw    (157.67,114.67) -- (157.67,161.67) ;
\draw [shift={(157.67,164.67)}, rotate = 270] [fill={rgb, 255:red, 0; green, 0; blue, 0 }  ][line width=0.08]  [draw opacity=0] (8.93,-4.29) -- (0,0) -- (8.93,4.29) -- cycle    ;
\draw    (499.4,40.07) -- (499.4,73.07) -- (499.4,87.07) ;
\draw [shift={(499.4,90.07)}, rotate = 270] [fill={rgb, 255:red, 0; green, 0; blue, 0 }  ][line width=0.08]  [draw opacity=0] (8.93,-4.29) -- (0,0) -- (8.93,4.29) -- cycle    ;
\draw    (500.07,114.4) -- (500.07,161.4) ;
\draw [shift={(500.07,164.4)}, rotate = 270] [fill={rgb, 255:red, 0; green, 0; blue, 0 }  ][line width=0.08]  [draw opacity=0] (8.93,-4.29) -- (0,0) -- (8.93,4.29) -- cycle    ;
\draw [color={rgb, 255:red, 208; green, 2; blue, 27 }  ,draw opacity=1 ]   (268,32.01) -- (387.82,32.59) ;
\draw [shift={(390.82,32.6)}, rotate = 180.27] [fill={rgb, 255:red, 208; green, 2; blue, 27 }  ,fill opacity=1 ][line width=0.08]  [draw opacity=0] (8.93,-4.29) -- (0,0) -- (8.93,4.29) -- cycle    ;
\draw [shift={(265,32)}, rotate = 0.27] [fill={rgb, 255:red, 208; green, 2; blue, 27 }  ,fill opacity=1 ][line width=0.08]  [draw opacity=0] (8.93,-4.29) -- (0,0) -- (8.93,4.29) -- cycle    ;
\draw [color={rgb, 255:red, 208; green, 2; blue, 27 }  ,draw opacity=1 ]   (268.5,101.61) -- (388.32,102.19) ;
\draw [shift={(391.32,102.2)}, rotate = 180.27] [fill={rgb, 255:red, 208; green, 2; blue, 27 }  ,fill opacity=1 ][line width=0.08]  [draw opacity=0] (8.93,-4.29) -- (0,0) -- (8.93,4.29) -- cycle    ;
\draw [shift={(265.5,101.6)}, rotate = 0.27] [fill={rgb, 255:red, 208; green, 2; blue, 27 }  ,fill opacity=1 ][line width=0.08]  [draw opacity=0] (8.93,-4.29) -- (0,0) -- (8.93,4.29) -- cycle    ;
\draw [color={rgb, 255:red, 208; green, 2; blue, 27 }  ,draw opacity=1 ]   (269.51,176.81) -- (389.33,177.39) ;
\draw [shift={(392.33,177.4)}, rotate = 180.27] [fill={rgb, 255:red, 208; green, 2; blue, 27 }  ,fill opacity=1 ][line width=0.08]  [draw opacity=0] (8.93,-4.29) -- (0,0) -- (8.93,4.29) -- cycle    ;
\draw [shift={(266.51,176.8)}, rotate = 0.27] [fill={rgb, 255:red, 208; green, 2; blue, 27 }  ,fill opacity=1 ][line width=0.08]  [draw opacity=0] (8.93,-4.29) -- (0,0) -- (8.93,4.29) -- cycle    ;

\draw (64,16.33) node [anchor=north west][inner sep=0.75pt]   [align=left] {Conjecture \ref{conjecture1} finite-version};
\draw (64,91.67) node [anchor=north west][inner sep=0.75pt]   [align=left] {Conjecture \ref{conjecture2} finite-version};
\draw (66,167) node [anchor=north west][inner sep=0.75pt]   [align=left] {Conjecture \ref{conjecture3} finite-version};
\draw (162.67,53.33) node [anchor=north west][inner sep=0.75pt]   [align=left] {};
\draw (402,16.33) node [anchor=north west][inner sep=0.75pt]   [align=left] {Conjecture \ref{conjecture1} infinite-version};
\draw (402,90.33) node [anchor=north west][inner sep=0.75pt]   [align=left] {Conjecture \ref{conjecture2} infinite-version};
\draw (400.67,167) node [anchor=north west][inner sep=0.75pt]   [align=left] {Conjecture \ref{conjecture3} infinite-version};
\draw (163.33,128.67) node [anchor=north west][inner sep=0.75pt]   [align=left] {\cite{perfectmatchingswithrestricted,MACAJOVA2005112}};
\draw (509.33,130.67) node [anchor=north west][inner sep=0.75pt]   [align=left] {Corollary \ref{cor1}};
\draw (508.67,52.67) node [anchor=north west][inner sep=0.75pt]   [align=left] {};
\draw (284.67,13) node [anchor=north west][inner sep=0.75pt]   [align=left] {\textcolor[rgb]{0.82,0.01,0.11}{Theorem \ref{thm2}}};
\draw (284.67,83) node [anchor=north west][inner sep=0.75pt]   [align=left] {\textcolor[rgb]{0.82,0.01,0.11}{Theorem \ref{thm3}}};
\draw (283.33,157) node [anchor=north west][inner sep=0.75pt]   [align=left] {\textcolor[rgb]{0.82,0.01,0.11}{Theorem \ref{thm4}}};

\end{tikzpicture}
    \caption{Diagram showing the implications among the conjectures.    }
    \label{fig5}
\end{figure}
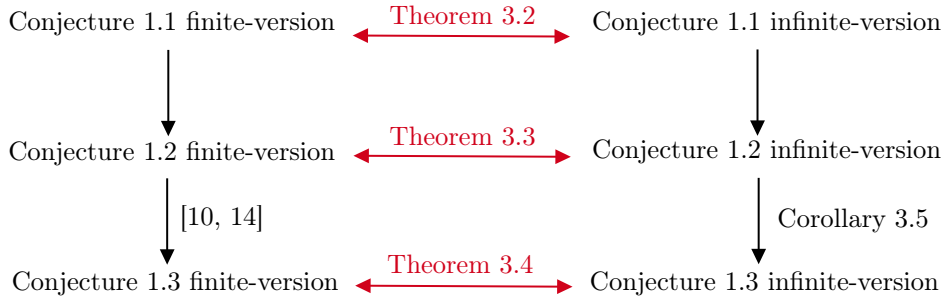

Although the technique of $F$-limits is a compactness result, we will require approximation lemmas for cubic and bridgeless graphs that preserve both of these properties. This will enable us to apply compactness in a manner that guarantees the resulting limit of the matchings satisfies the desired properties. This development is carried out in Section \ref{section2}, which presents the approximation lemmas \ref{thm1}, \ref{lemma2}, and \ref{lemma3}. Lemma \ref{thm1}, in particular, will allow us to view cubic bridgeless graphs as $F$-limits of finite graphs sharing these same properties. However, the proof presented is not the most direct, simple, or shortest possible. This is due to the fact that the conclusion of Conjecture \ref{conjecture3} concerns odd cuts. To guarantee this property and avoid the parity issues that arise, the construction had to be carried out exactly as presented. Meanwhile, Lemma \ref{lemma3} is formulated within the framework of edge covers in order to generalize the properties required of perfect matchings in Conjectures \ref{conjecture1} and \ref{conjecture2}.

In Section \ref{section3}, we present the main results of the paper, which establish the connections between the finite and infinite versions of the conjectures. We begin with Theorem \ref{thm5}, which is stated in the language of covers to generalize the properties required by Conjectures \ref{conjecture1} and \ref{conjecture2}.

\begin{thm5}
    Let $G$ be an infinite, cubic, bridgeless, connected graph, and let $p,k \in \mathbb{N}_{>0}$ with $p \leq k$. If every finite, cubic, bridgeless, connected graph has a $(k,p)$-PM cover ($(k,p)^*$-PM cover), then $G$ also has a $(k,p)$-PM cover ($(k,p)^*$-PM cover).
\end{thm5}

As a direct application, we obtain Theorems \ref{thm2} and \ref{thm3}, which establish, respectively, that the finite version of the Berge-Fulkerson Conjecture is equivalent to its infinite version, and that the finite version of the Fan-Raspaud Conjecture is equivalent to its infinite version.

\begin{thm2}
    Conjecture~\ref{conjecture1} holds for finite graphs if and only if it holds for infinite graphs.
\end{thm2}

\begin{thm3}
    Conjecture~\ref{conjecture2} holds for finite graphs if and only if it holds for infinite graphs.
\end{thm3}

In the case of the M\'a\v{c}ajov\'a and \v{S}koviera Conjecture, the result that the finite case implies the infinite case does not follow from Theorem \ref{thm5}. This is due to the structural nature of the conjecture's requirement, namely, the odd cuts. Odd cuts become a significant obstacle here because an arbitrary approximation could end up obscuring the parities of the cuts during the finite processes, ultimately yielding an $F$-limit that fails to respect the odd cut condition. For this reason, the proof of Lemma \ref{thm1} introduces an approximation construction that is considerably more technical than what would be required to prove the statement alone. This is done precisely to achieve a construction that obeys the parities of the cuts throughout the finite approximation. In this way, we make a more complex application of compactness arguments, respecting the graph cut combinatorics, to obtain Theorem \ref{thm4}, which provides the finite-to-infinite equivalence for Conjecture \ref{conjecture3}.

\begin{thm4}
    If Conjecture~\ref{conjecture3} holds for finite graphs if and only if it holds for infinite graphs.
\end{thm4}

Note now that when a cut separates two infinite parts of a graph, it is impossible to use classical parity arguments over the cut. This is because both sides are infinite, making it impossible to apply vertex or edge counting arguments on either side. Furthermore, flow-based approaches may run into topological difficulties for infinite flows. Consequently, we cannot obtain in a direct and clean manner, by arguments analogous to those used in the finite case, that the infinite version of Conjecture \ref{conjecture2} implies the infinite version of Conjecture \ref{conjecture3}. To overcome these difficulties, we combine our established results---Theorem \ref{thm3} and Theorem \ref{thm4}---with the implication already verified for the finite case. Thus, we smoothly obtain Corollary \ref{cor1}.

\begin{cor1}
    If the Fan-Raspaud conjecture holds for infinite graphs, then so does the conjecture of M\'{a}\v{c}ajov\'{a} and \v{S}koviera.
\end{cor1}

We conclude the paper with Corollary \ref{thm6}, which extends the validity of Conjecture \ref{conjecture4} to infinite graphs by means of the constructions developed herein.

\begin{thm6}
    Every infinite, cubic, bridgeless graph $G$ has two perfect matchings $M_1$, $M_2$ so that $G \setminus (M_1\cup M_2)$ is bipartite.
\end{thm6}

\section{Approximations lemmas}\label{section2}
\paragraph{}

In this section, we present technical lemmas about finite approximations and the $F$-limits of these approximations, which will be applied to the results of the next section. The notation and terminology used follow reference~\cite{diestelbook}. Before the lemmas, we formally present the definition of the $F$-limit.

\begin{defn}
    We say that $F\subset \wp(\mathbb{N})$ is a \textbf{filter} over $\mathbb{N}$ if:
     \begin{enumerate}[(a)]
        \item $\emptyset \notin F$;
        \item If $a,b\in F$ then $a\cap b \in F$;
        \item If $a\in F$ and $b\supset a$ then $b\in F$;
      \end{enumerate}
    We say that $F$ is an \textbf{ultrafilter} if $F$ is maximal (that is, if $G\supset F$ is a filter then $G=F$). We also say that an ultrafilter $F$ is \textbf{non-principal} if it does not contain finite subsets of $\mathbb{N}$ (see e.g. \cite{Jechbook}). 
\end{defn}

We will use the following fact about ultrafilters.

\begin{prop}[see e.g. \cite{Jechbook} Chapter 7]\label{ultrafilter}
    Let $F$ be a filter on $\mathbb{N}$. The following statements are equivalent:
    \begin{enumerate}
        \item $F$ is an ultrafilter;
        \item If $A_0 \cup \dots \cup A_n \in F$ is a finite union, then there exists $i \in \{0, 1, \dots, n\}$ such that $A_i \in F$.
    \end{enumerate}
\end{prop}

\begin{defn}
    Let $F$ be a non-principal ultrafilter over $\mathbb{N}$, let $G$ be a graph and let $\langle H_n\rangle_{n\in\mathbb{N}}$ be a sequence of subgraphs of $G$. 
    We say that a subgraph $H\subset G$ is the \textbf{$F$-limit} of $\langle H_n\rangle_{n\in\mathbb{N}}$ if $H=(V_H,E_H)$, where
    $$V_H=\lbrace v\in V(G) : \lbrace n : v\in V(H_n) \rbrace \in F \rbrace ,$$
    $$E_H=\lbrace e\in E(G) : \lbrace n : e\in E(H_n) \rbrace \in F \rbrace.$$
    
    In this case, we say that the sequence $\langle H_n\rangle_{n\in\mathbb{N}}$ is \textbf{convergent} and \textbf{converges} to $H$.
    
\end{defn}

Given an infinite, cubic, and bridgeless graph $G$, we present the construction of a sequence of finite, cubic, and bridgeless graphs $\langle H_i\rangle_{i\in\mathbb{N}}$ such that $G$ is the $F$-limit of $\langle H_n\rangle_{n\in\mathbb{N}}$. The construction presented below is neither unique nor the simplest; however, it is designed to be applicable in the proofs of Theorems \ref{thm2}, \ref{thm3}, and \ref{thm4}.

\begin{lemma}\label{thm1}
    Let $G$ be an infinite, cubic, bridgeless connected graph. Then $G$ is the $F$-limit of a sequence of finite, bridgeless cubic graphs. 
\end{lemma}

\begin{proof}
    Let $E(G)=\lbrace e_i\rbrace_{i\in\mathbb{N}}$. Consider $\langle G_i : i\in\mathbb{N}\rangle$ an ear decomposition for $G$ constructed as follows: let $G_0$ be a cycle containing the edge $e_0$. For each $n \in \mathbb{N}$, we define the subgraph $G_{n+1}$ as follows: if $e_{n+1}\notin E(G_{n})$, then $G_{n+1}=G_n\cup P_{n+1}$, where $P_{n+1}$ is an ear containing the edge $e_{n+1}$, as shown in Figure~\ref{fig1}; if $e_{n+1}\in E(G_n)$, then $G_{n+1}=G_{n}$. Note that each added ear is always open because $G$ is cubic. Also, note that each $G_n$ is bridgeless. Furthermore, $G_n\subset G_{n+1}$ and $G=\bigcup_{n\in\mathbb{N}}G_n$.

    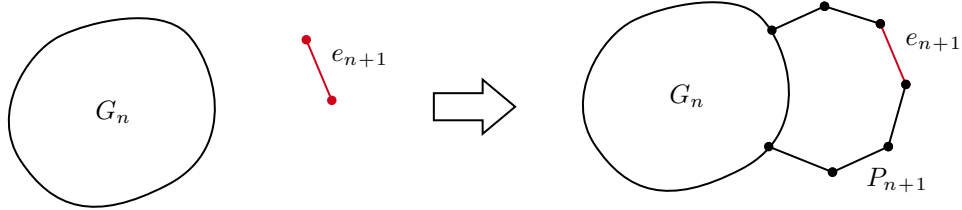
\begin{figure}[ht]
        \centering

\tikzset{every picture/.style={line width=0.75pt}} 

\begin{tikzpicture}[x=0.6pt,y=0.6pt,yscale=-1,xscale=1]

\draw   (64,28) .. controls (84,18) and (127,9) .. (144,26) .. controls (161,43) and (178,97) .. (132,121) .. controls (86,145) and (60,136) .. (40,106) .. controls (20,76) and (44,38) .. (64,28) -- cycle ;
\draw [color={rgb, 255:red, 208; green, 2; blue, 27 }  ,draw opacity=1 ]   (219,30) -- (235,68) ;
\draw [shift={(235,68)}, rotate = 67.17] [color={rgb, 255:red, 208; green, 2; blue, 27 }  ,draw opacity=1 ][fill={rgb, 255:red, 208; green, 2; blue, 27 }  ,fill opacity=1 ][line width=0.75]      (0, 0) circle [x radius= 2.34, y radius= 2.34]   ;
\draw [shift={(219,30)}, rotate = 67.17] [color={rgb, 255:red, 208; green, 2; blue, 27 }  ,draw opacity=1 ][fill={rgb, 255:red, 208; green, 2; blue, 27 }  ,fill opacity=1 ][line width=0.75]      (0, 0) circle [x radius= 2.34, y radius= 2.34]   ;
\draw   (424,18) .. controls (444,8) and (487,-1) .. (504,16) .. controls (521,33) and (538,87) .. (492,111) .. controls (446,135) and (420,126) .. (400,96) .. controls (380,66) and (404,28) .. (424,18) -- cycle ;
\draw [color={rgb, 255:red, 208; green, 2; blue, 27 }  ,draw opacity=1 ]   (579,20) -- (595,58) ;
\draw [shift={(595,58)}, rotate = 67.17] [color={rgb, 255:red, 208; green, 2; blue, 27 }  ,draw opacity=1 ][fill={rgb, 255:red, 208; green, 2; blue, 27 }  ,fill opacity=1 ][line width=0.75]      (0, 0) circle [x radius= 2.34, y radius= 2.34]   ;
\draw [shift={(579,20)}, rotate = 67.17] [color={rgb, 255:red, 208; green, 2; blue, 27 }  ,draw opacity=1 ][fill={rgb, 255:red, 208; green, 2; blue, 27 }  ,fill opacity=1 ][line width=0.75]      (0, 0) circle [x radius= 2.34, y radius= 2.34]   ;
\draw    (509,97) -- (549,113) ;
\draw [shift={(549,113)}, rotate = 21.8] [color={rgb, 255:red, 0; green, 0; blue, 0 }  ][fill={rgb, 255:red, 0; green, 0; blue, 0 }  ][line width=0.75]      (0, 0) circle [x radius= 2.34, y radius= 2.34]   ;
\draw [shift={(509,97)}, rotate = 21.8] [color={rgb, 255:red, 0; green, 0; blue, 0 }  ][fill={rgb, 255:red, 0; green, 0; blue, 0 }  ][line width=0.75]      (0, 0) circle [x radius= 2.34, y radius= 2.34]   ;
\draw    (549,113) -- (584,97) ;
\draw [shift={(584,97)}, rotate = 335.43] [color={rgb, 255:red, 0; green, 0; blue, 0 }  ][fill={rgb, 255:red, 0; green, 0; blue, 0 }  ][line width=0.75]      (0, 0) circle [x radius= 2.34, y radius= 2.34]   ;
\draw [shift={(549,113)}, rotate = 335.43] [color={rgb, 255:red, 0; green, 0; blue, 0 }  ][fill={rgb, 255:red, 0; green, 0; blue, 0 }  ][line width=0.75]      (0, 0) circle [x radius= 2.34, y radius= 2.34]   ;
\draw    (584,97) -- (595,58) ;
\draw [shift={(595,58)}, rotate = 285.75] [color={rgb, 255:red, 0; green, 0; blue, 0 }  ][fill={rgb, 255:red, 0; green, 0; blue, 0 }  ][line width=0.75]      (0, 0) circle [x radius= 2.34, y radius= 2.34]   ;
\draw [shift={(584,97)}, rotate = 285.75] [color={rgb, 255:red, 0; green, 0; blue, 0 }  ][fill={rgb, 255:red, 0; green, 0; blue, 0 }  ][line width=0.75]      (0, 0) circle [x radius= 2.34, y radius= 2.34]   ;
\draw    (544,9) -- (579,20) ;
\draw [shift={(579,20)}, rotate = 17.45] [color={rgb, 255:red, 0; green, 0; blue, 0 }  ][fill={rgb, 255:red, 0; green, 0; blue, 0 }  ][line width=0.75]      (0, 0) circle [x radius= 2.34, y radius= 2.34]   ;
\draw [shift={(544,9)}, rotate = 17.45] [color={rgb, 255:red, 0; green, 0; blue, 0 }  ][fill={rgb, 255:red, 0; green, 0; blue, 0 }  ][line width=0.75]      (0, 0) circle [x radius= 2.34, y radius= 2.34]   ;
\draw    (511,24) -- (544,9) ;
\draw [shift={(544,9)}, rotate = 335.56] [color={rgb, 255:red, 0; green, 0; blue, 0 }  ][fill={rgb, 255:red, 0; green, 0; blue, 0 }  ][line width=0.75]      (0, 0) circle [x radius= 2.34, y radius= 2.34]   ;
\draw [shift={(511,24)}, rotate = 335.56] [color={rgb, 255:red, 0; green, 0; blue, 0 }  ][fill={rgb, 255:red, 0; green, 0; blue, 0 }  ][line width=0.75]      (0, 0) circle [x radius= 2.34, y radius= 2.34]   ;
\draw   (299,63.5) -- (329.6,63.5) -- (329.6,55) -- (350,72) -- (329.6,89) -- (329.6,80.5) -- (299,80.5) -- cycle ;

\draw (85,66.4) node [anchor=north west][inner sep=0.75pt]    {$G_{n}$};
\draw (233,34.4) node [anchor=north west][inner sep=0.75pt]    {$e_{n}{}_{+}{}_{1}$};
\draw (445,56.4) node [anchor=north west][inner sep=0.75pt]    {$G_{n}$};
\draw (593,24.4) node [anchor=north west][inner sep=0.75pt]    {$e_{n}{}_{+}{}_{1}$};
\draw (568.5,108.4) node [anchor=north west][inner sep=0.75pt]    {$P_{n}{}_{+}{}_{1}$};

\end{tikzpicture}
        \caption{Ear decomposition}
        \label{fig1}
    \end{figure}

Note that all vertices of each $G_n$ have degree 2 or 3, since $\langle G_n : n\in\mathbb{N}\rangle$ is an open ear decomposition. For each $i\in\mathbb{N}$, define $X_i=\lbrace v\in V(G_i) : d_{G_i}(v)=2\rbrace$. Let $m_i\in\mathbb{N}$ be the smallest natural number such that $d_{G_{m_i}}(v)=3$ for all $v\in X_i$. Let $\mathcal{P}_i$ be the set of paths $P\subset G_{m_i}$ that are maximal with respect to the following property: all internal vertices of $P$ have degree 2 in $G_{m_i}$.

Observe that if $P, Q\in \mathcal{P}_i$ and $P\neq Q$, then these two paths are internally disjoint. Thus, every vertex $v\in V(G_{m_i})$ with $d_{G_{m_i}}(v)=2$ lies on a unique path $P\in \mathcal{P}_i$. Consider the graph $H_i'$ obtained from $G_{m_i}$ by contracting each path $P\in \mathcal{P}_i$ into an edge, as shown in Figure~\ref{fig2}. In this way, each $Q\in \mathcal{P}_i$ is contracted into an edge $e_Q\in E(H_i')$, where the endpoints of $Q$ become the endpoints of $e_Q$. Note that $H_i'$ is cubic. Indeed, let $v\in V(H_i')$; then $v\in V(G_{m_i})$, since $V(H_i')\subset V(G_{m_i})$. Since $v\in V(H_i')$, then $d_{G_{m_i}}(v)=3$; otherwise, we would have $d_{G_{m_i}}(v)=2$, and there would exist $P\in \mathcal{P}_i$ such that $v$ is an internal vertex of $P$. Since $H_i'$ was obtained from $G_{m_i}$ by contracting the paths in $\mathcal{P}_i$, we would have $v\notin V(H_i')$, which is a contradiction. Therefore, $d_{G_{m_i}}(v)=3$, which implies $d_{H_i'}(v)=3$.

\begin{figure}[ht]
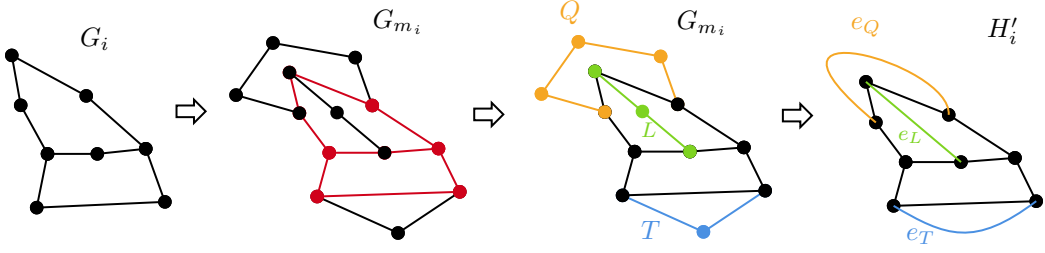

    \centering

\tikzset{every picture/.style={line width=0.75pt}} 


    \caption{Construction of the $G_{m_i}$ and $H_i'$}
    \label{fig2}
\end{figure}

Now, we will eliminate the parallel edges that may have arisen in $H_i'$ in the following way:

\begin{claim}\label{claim0}
For every $i\in\mathbb{N}$ and any $u,v\in V(H_i')$ there are at most two edges in $H_i'$ joining $u$ and $v$.
  
\end{claim}
\begin{proof}
Indeed, suppose there exist $u, v \in V(H_i')$ such that $e_1=uv, e_2=uv, e_3=uv \in E(H_i')$ with $e_1 \neq e_2 \neq e_3$. Since $H_i'$ is cubic, then $H_i'$ is formed precisely by the vertices $u, v$ and the edges $e_1, e_2, e_3$, as shown in Figure \ref{fig3} (a). Since $G_{m_i}$ has no parallel edges, at least two of the edges $e_1, e_2, e_3$ must have originated from contracted paths in $G_{m_i}$. Suppose, without loss of generality, that the edges $e_1$ and $e_2$ were contracted from paths denoted by $Q_{e_1}, Q_{e_2} \subset G_{m_i}$. Thus, $G_{m_i}$ consists of either the paths $Q_{e_1}$ and $Q_{e_2}$ and an edge $e_3$ between $u$ and $v$, as in Figure \ref{fig3} (b), or the paths $Q_{e_1}, Q_{e_2}$ and a path $Q_{e_3}$ between $u$ and $v$, as in Figure \ref{fig3} (c).

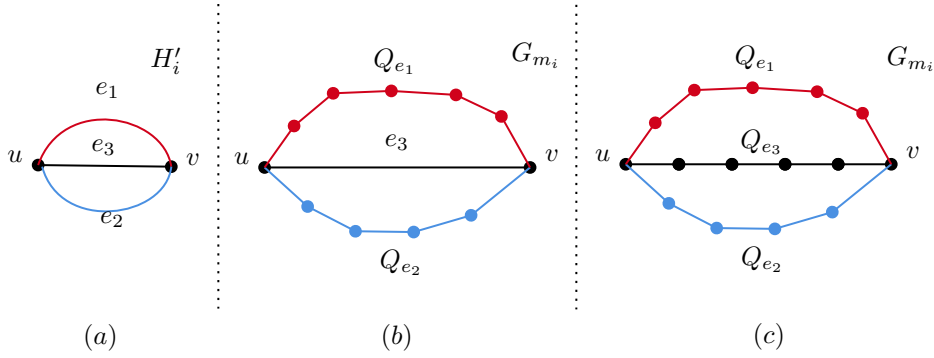
\begin{figure}[ht]
    \centering

\tikzset{every picture/.style={line width=0.75pt}} 

\begin{tikzpicture}[x=0.6pt,y=0.6pt,yscale=-1,xscale=1]

\draw    (64.33,107.67) -- (147.83,108.67) ;
\draw [shift={(147.83,108.67)}, rotate = 0.69] [color={rgb, 255:red, 0; green, 0; blue, 0 }  ][fill={rgb, 255:red, 0; green, 0; blue, 0 }  ][line width=0.75]      (0, 0) circle [x radius= 3.35, y radius= 3.35]   ;
\draw [shift={(64.33,107.67)}, rotate = 0.69] [color={rgb, 255:red, 0; green, 0; blue, 0 }  ][fill={rgb, 255:red, 0; green, 0; blue, 0 }  ][line width=0.75]      (0, 0) circle [x radius= 3.35, y radius= 3.35]   ;
\draw [color={rgb, 255:red, 208; green, 2; blue, 27 }  ,draw opacity=1 ]   (64.33,107.67) .. controls (79.83,63.67) and (142.83,75.67) .. (147.83,108.67) ;
\draw [color={rgb, 255:red, 74; green, 144; blue, 226 }  ,draw opacity=1 ]   (66.83,108.67) .. controls (78.33,150.67) and (144.33,141.17) .. (147.83,108.67) ;
\draw  [dash pattern={on 0.84pt off 2.51pt}]  (177.33,6.67) -- (177.33,197.67) ;
\draw    (206.33,109.17) -- (372.83,109.17) ;
\draw [shift={(372.83,109.17)}, rotate = 0] [color={rgb, 255:red, 0; green, 0; blue, 0 }  ][fill={rgb, 255:red, 0; green, 0; blue, 0 }  ][line width=0.75]      (0, 0) circle [x radius= 3.35, y radius= 3.35]   ;
\draw [shift={(206.33,109.17)}, rotate = 0] [color={rgb, 255:red, 0; green, 0; blue, 0 }  ][fill={rgb, 255:red, 0; green, 0; blue, 0 }  ][line width=0.75]      (0, 0) circle [x radius= 3.35, y radius= 3.35]   ;
\draw [color={rgb, 255:red, 74; green, 144; blue, 226 }  ,draw opacity=1 ]   (206.33,109.17) -- (233.33,133.67) ;
\draw [shift={(233.33,133.67)}, rotate = 42.22] [color={rgb, 255:red, 74; green, 144; blue, 226 }  ,draw opacity=1 ][fill={rgb, 255:red, 74; green, 144; blue, 226 }  ,fill opacity=1 ][line width=0.75]      (0, 0) circle [x radius= 3.35, y radius= 3.35]   ;
\draw [color={rgb, 255:red, 74; green, 144; blue, 226 }  ,draw opacity=1 ]   (233.33,133.67) -- (263.33,149.17) ;
\draw [shift={(263.33,149.17)}, rotate = 27.32] [color={rgb, 255:red, 74; green, 144; blue, 226 }  ,draw opacity=1 ][fill={rgb, 255:red, 74; green, 144; blue, 226 }  ,fill opacity=1 ][line width=0.75]      (0, 0) circle [x radius= 3.35, y radius= 3.35]   ;
\draw [color={rgb, 255:red, 74; green, 144; blue, 226 }  ,draw opacity=1 ]   (263.33,149.17) -- (299.83,149.67) ;
\draw [shift={(299.83,149.67)}, rotate = 0.78] [color={rgb, 255:red, 74; green, 144; blue, 226 }  ,draw opacity=1 ][fill={rgb, 255:red, 74; green, 144; blue, 226 }  ,fill opacity=1 ][line width=0.75]      (0, 0) circle [x radius= 3.35, y radius= 3.35]   ;
\draw [color={rgb, 255:red, 74; green, 144; blue, 226 }  ,draw opacity=1 ]   (299.83,149.67) -- (335.83,139.17) ;
\draw [shift={(335.83,139.17)}, rotate = 343.74] [color={rgb, 255:red, 74; green, 144; blue, 226 }  ,draw opacity=1 ][fill={rgb, 255:red, 74; green, 144; blue, 226 }  ,fill opacity=1 ][line width=0.75]      (0, 0) circle [x radius= 3.35, y radius= 3.35]   ;
\draw [color={rgb, 255:red, 74; green, 144; blue, 226 }  ,draw opacity=1 ]   (335.83,139.17) -- (372.83,109.17) ;
\draw [color={rgb, 255:red, 208; green, 2; blue, 27 }  ,draw opacity=1 ]   (206.33,109.17) -- (224.83,83.17) ;
\draw [shift={(224.83,83.17)}, rotate = 305.43] [color={rgb, 255:red, 208; green, 2; blue, 27 }  ,draw opacity=1 ][fill={rgb, 255:red, 208; green, 2; blue, 27 }  ,fill opacity=1 ][line width=0.75]      (0, 0) circle [x radius= 3.35, y radius= 3.35]   ;
\draw [color={rgb, 255:red, 208; green, 2; blue, 27 }  ,draw opacity=1 ]   (224.83,83.17) -- (249.33,62.67) ;
\draw [shift={(249.33,62.67)}, rotate = 320.08] [color={rgb, 255:red, 208; green, 2; blue, 27 }  ,draw opacity=1 ][fill={rgb, 255:red, 208; green, 2; blue, 27 }  ,fill opacity=1 ][line width=0.75]      (0, 0) circle [x radius= 3.35, y radius= 3.35]   ;
\draw [color={rgb, 255:red, 208; green, 2; blue, 27 }  ,draw opacity=1 ]   (249.33,62.67) -- (285.83,61.17) ;
\draw [shift={(285.83,61.17)}, rotate = 357.65] [color={rgb, 255:red, 208; green, 2; blue, 27 }  ,draw opacity=1 ][fill={rgb, 255:red, 208; green, 2; blue, 27 }  ,fill opacity=1 ][line width=0.75]      (0, 0) circle [x radius= 3.35, y radius= 3.35]   ;
\draw [color={rgb, 255:red, 208; green, 2; blue, 27 }  ,draw opacity=1 ]   (285.83,61.17) -- (326.33,63.67) ;
\draw [shift={(326.33,63.67)}, rotate = 3.53] [color={rgb, 255:red, 208; green, 2; blue, 27 }  ,draw opacity=1 ][fill={rgb, 255:red, 208; green, 2; blue, 27 }  ,fill opacity=1 ][line width=0.75]      (0, 0) circle [x radius= 3.35, y radius= 3.35]   ;
\draw [color={rgb, 255:red, 208; green, 2; blue, 27 }  ,draw opacity=1 ]   (326.33,63.67) -- (354.83,77.17) ;
\draw [shift={(354.83,77.17)}, rotate = 25.35] [color={rgb, 255:red, 208; green, 2; blue, 27 }  ,draw opacity=1 ][fill={rgb, 255:red, 208; green, 2; blue, 27 }  ,fill opacity=1 ][line width=0.75]      (0, 0) circle [x radius= 3.35, y radius= 3.35]   ;
\draw [color={rgb, 255:red, 208; green, 2; blue, 27 }  ,draw opacity=1 ]   (354.83,77.17) -- (372.83,109.17) ;
\draw    (432.83,107.17) -- (466.13,107.17) ;
\draw [shift={(466.13,107.17)}, rotate = 0] [color={rgb, 255:red, 0; green, 0; blue, 0 }  ][fill={rgb, 255:red, 0; green, 0; blue, 0 }  ][line width=0.75]      (0, 0) circle [x radius= 3.35, y radius= 3.35]   ;
\draw [shift={(432.83,107.17)}, rotate = 0] [color={rgb, 255:red, 0; green, 0; blue, 0 }  ][fill={rgb, 255:red, 0; green, 0; blue, 0 }  ][line width=0.75]      (0, 0) circle [x radius= 3.35, y radius= 3.35]   ;
\draw    (466.13,107.17) -- (499.43,107.17) ;
\draw [shift={(499.43,107.17)}, rotate = 360] [color={rgb, 255:red, 0; green, 0; blue, 0 }  ][fill={rgb, 255:red, 0; green, 0; blue, 0 }  ][line width=0.75]      (0, 0) circle [x radius= 3.35, y radius= 3.35]   ;
\draw [shift={(466.13,107.17)}, rotate = 360] [color={rgb, 255:red, 0; green, 0; blue, 0 }  ][fill={rgb, 255:red, 0; green, 0; blue, 0 }  ][line width=0.75]      (0, 0) circle [x radius= 3.35, y radius= 3.35]   ;
\draw    (499.43,107.17) -- (532.73,107.17) ;
\draw [shift={(532.73,107.17)}, rotate = 0] [color={rgb, 255:red, 0; green, 0; blue, 0 }  ][fill={rgb, 255:red, 0; green, 0; blue, 0 }  ][line width=0.75]      (0, 0) circle [x radius= 3.35, y radius= 3.35]   ;
\draw [shift={(499.43,107.17)}, rotate = 0] [color={rgb, 255:red, 0; green, 0; blue, 0 }  ][fill={rgb, 255:red, 0; green, 0; blue, 0 }  ][line width=0.75]      (0, 0) circle [x radius= 3.35, y radius= 3.35]   ;
\draw    (532.73,107.17) -- (566.03,107.17) ;
\draw [shift={(566.03,107.17)}, rotate = 0] [color={rgb, 255:red, 0; green, 0; blue, 0 }  ][fill={rgb, 255:red, 0; green, 0; blue, 0 }  ][line width=0.75]      (0, 0) circle [x radius= 3.35, y radius= 3.35]   ;
\draw [shift={(532.73,107.17)}, rotate = 0] [color={rgb, 255:red, 0; green, 0; blue, 0 }  ][fill={rgb, 255:red, 0; green, 0; blue, 0 }  ][line width=0.75]      (0, 0) circle [x radius= 3.35, y radius= 3.35]   ;
\draw    (566.03,107.17) -- (599.33,107.17) ;
\draw [shift={(599.33,107.17)}, rotate = 0] [color={rgb, 255:red, 0; green, 0; blue, 0 }  ][fill={rgb, 255:red, 0; green, 0; blue, 0 }  ][line width=0.75]      (0, 0) circle [x radius= 3.35, y radius= 3.35]   ;
\draw [shift={(566.03,107.17)}, rotate = 0] [color={rgb, 255:red, 0; green, 0; blue, 0 }  ][fill={rgb, 255:red, 0; green, 0; blue, 0 }  ][line width=0.75]      (0, 0) circle [x radius= 3.35, y radius= 3.35]   ;
\draw [color={rgb, 255:red, 74; green, 144; blue, 226 }  ,draw opacity=1 ]   (432.83,107.17) -- (459.83,131.67) ;
\draw [shift={(459.83,131.67)}, rotate = 42.22] [color={rgb, 255:red, 74; green, 144; blue, 226 }  ,draw opacity=1 ][fill={rgb, 255:red, 74; green, 144; blue, 226 }  ,fill opacity=1 ][line width=0.75]      (0, 0) circle [x radius= 3.35, y radius= 3.35]   ;
\draw [color={rgb, 255:red, 74; green, 144; blue, 226 }  ,draw opacity=1 ]   (459.83,131.67) -- (489.83,147.17) ;
\draw [shift={(489.83,147.17)}, rotate = 27.32] [color={rgb, 255:red, 74; green, 144; blue, 226 }  ,draw opacity=1 ][fill={rgb, 255:red, 74; green, 144; blue, 226 }  ,fill opacity=1 ][line width=0.75]      (0, 0) circle [x radius= 3.35, y radius= 3.35]   ;
\draw [color={rgb, 255:red, 74; green, 144; blue, 226 }  ,draw opacity=1 ]   (489.83,147.17) -- (526.33,147.67) ;
\draw [shift={(526.33,147.67)}, rotate = 0.78] [color={rgb, 255:red, 74; green, 144; blue, 226 }  ,draw opacity=1 ][fill={rgb, 255:red, 74; green, 144; blue, 226 }  ,fill opacity=1 ][line width=0.75]      (0, 0) circle [x radius= 3.35, y radius= 3.35]   ;
\draw [color={rgb, 255:red, 74; green, 144; blue, 226 }  ,draw opacity=1 ]   (526.33,147.67) -- (562.33,137.17) ;
\draw [shift={(562.33,137.17)}, rotate = 343.74] [color={rgb, 255:red, 74; green, 144; blue, 226 }  ,draw opacity=1 ][fill={rgb, 255:red, 74; green, 144; blue, 226 }  ,fill opacity=1 ][line width=0.75]      (0, 0) circle [x radius= 3.35, y radius= 3.35]   ;
\draw [color={rgb, 255:red, 74; green, 144; blue, 226 }  ,draw opacity=1 ]   (562.33,137.17) -- (599.33,107.17) ;
\draw [color={rgb, 255:red, 208; green, 2; blue, 27 }  ,draw opacity=1 ]   (432.83,107.17) -- (451.33,81.17) ;
\draw [shift={(451.33,81.17)}, rotate = 305.43] [color={rgb, 255:red, 208; green, 2; blue, 27 }  ,draw opacity=1 ][fill={rgb, 255:red, 208; green, 2; blue, 27 }  ,fill opacity=1 ][line width=0.75]      (0, 0) circle [x radius= 3.35, y radius= 3.35]   ;
\draw [color={rgb, 255:red, 208; green, 2; blue, 27 }  ,draw opacity=1 ]   (451.33,81.17) -- (475.83,60.67) ;
\draw [shift={(475.83,60.67)}, rotate = 320.08] [color={rgb, 255:red, 208; green, 2; blue, 27 }  ,draw opacity=1 ][fill={rgb, 255:red, 208; green, 2; blue, 27 }  ,fill opacity=1 ][line width=0.75]      (0, 0) circle [x radius= 3.35, y radius= 3.35]   ;
\draw [color={rgb, 255:red, 208; green, 2; blue, 27 }  ,draw opacity=1 ]   (475.83,60.67) -- (512.33,59.17) ;
\draw [shift={(512.33,59.17)}, rotate = 357.65] [color={rgb, 255:red, 208; green, 2; blue, 27 }  ,draw opacity=1 ][fill={rgb, 255:red, 208; green, 2; blue, 27 }  ,fill opacity=1 ][line width=0.75]      (0, 0) circle [x radius= 3.35, y radius= 3.35]   ;
\draw [color={rgb, 255:red, 208; green, 2; blue, 27 }  ,draw opacity=1 ]   (512.33,59.17) -- (552.83,61.67) ;
\draw [shift={(552.83,61.67)}, rotate = 3.53] [color={rgb, 255:red, 208; green, 2; blue, 27 }  ,draw opacity=1 ][fill={rgb, 255:red, 208; green, 2; blue, 27 }  ,fill opacity=1 ][line width=0.75]      (0, 0) circle [x radius= 3.35, y radius= 3.35]   ;
\draw [color={rgb, 255:red, 208; green, 2; blue, 27 }  ,draw opacity=1 ]   (552.83,61.67) -- (581.33,75.17) ;
\draw [shift={(581.33,75.17)}, rotate = 25.35] [color={rgb, 255:red, 208; green, 2; blue, 27 }  ,draw opacity=1 ][fill={rgb, 255:red, 208; green, 2; blue, 27 }  ,fill opacity=1 ][line width=0.75]      (0, 0) circle [x radius= 3.35, y radius= 3.35]   ;
\draw [color={rgb, 255:red, 208; green, 2; blue, 27 }  ,draw opacity=1 ]   (581.33,75.17) -- (599.33,107.17) ;
\draw  [dash pattern={on 0.84pt off 2.51pt}]  (401.83,7.17) -- (401.83,198.17) ;

\draw (43.33,96.57) node [anchor=north west][inner sep=0.75pt]    {$u$};
\draw (155.33,98.57) node [anchor=north west][inner sep=0.75pt]    {$v$};
\draw (98.33,54.07) node [anchor=north west][inner sep=0.75pt]    {$e_{1}$};
\draw (101.33,135.07) node [anchor=north west][inner sep=0.75pt]    {$e_{2}$};
\draw (95.83,90.07) node [anchor=north west][inner sep=0.75pt]    {$e_{3}$};
\draw (132.83,32.07) node [anchor=north west][inner sep=0.75pt]    {$H_{i} '$};
\draw (185.33,98.07) node [anchor=north west][inner sep=0.75pt]    {$u$};
\draw (380.33,96.07) node [anchor=north west][inner sep=0.75pt]    {$v$};
\draw (279.83,86.57) node [anchor=north west][inner sep=0.75pt]    {$e_{3}$};
\draw (276.83,159.07) node [anchor=north west][inner sep=0.75pt]    {$Q_{e_{2}}$};
\draw (272.83,33.07) node [anchor=north west][inner sep=0.75pt]    {$Q_{e_{1}}$};
\draw (411.83,96.07) node [anchor=north west][inner sep=0.75pt]    {$u$};
\draw (606.83,94.07) node [anchor=north west][inner sep=0.75pt]    {$v$};
\draw (503.33,157.07) node [anchor=north west][inner sep=0.75pt]    {$Q_{e_{2}}$};
\draw (502.83,83.57) node [anchor=north west][inner sep=0.75pt]    {$Q_{e_{3}}$};
\draw (499.33,31.07) node [anchor=north west][inner sep=0.75pt]    {$Q_{e_{1}}$};
\draw (90.67,205.73) node [anchor=north west][inner sep=0.75pt]    {$( a)$};
\draw (277.33,205.73) node [anchor=north west][inner sep=0.75pt]    {$( b)$};
\draw (510.67,205.07) node [anchor=north west][inner sep=0.75pt]    {$( c)$};
\draw (359.33,29.4) node [anchor=north west][inner sep=0.75pt]    {$G_{m_{i}}$};
\draw (594.67,31.4) node [anchor=north west][inner sep=0.75pt]    {$G_{m_{i}}$};

\end{tikzpicture}
    \caption{Parallel edges in $H_i'$}
    \label{fig3}
\end{figure}

However, $G_{m_i}$ cannot be as the graphs in Figure~\ref{fig3} (b) or (c). Indeed, recall that $G_i \subset G_{m_i}$ and, if $w \in V(G_i)$, then $d_{G_{m_i}}(w)=3$. Nevertheless, in a graph $G_{m_i}$ such as those in Figure~\ref{fig3} (b) or (c), the only vertices of degree 3 are $u$ and $v$. Thus, we would have $V(G_i) \subset \{u, v\}$. Consequently, $G_i$ could not contain any cycles, which is a contradiction, since $G_0 \subset G_i$ for $i \geq 0$ and $G_0$ is a cycle. This completes the proof of the claim.
\end{proof}

Thus, given $i\in\mathbb{N}$ and $u,v\in V(H_i')$, there are at most two edges in $H_i'$ connecting these two vertices, say $e_{uv}$ and $e'_{uv}$. We define the graph $H_i$ from $H_i'$ through the following process: for each pair of vertices $u,v\in V(H_i')$ forming double edges $e_{uv}$ and $e'_{uv}$, we subdivide them by adding vertex $w_{uv}$ to $e_{uv}$ and vertex $w'_{uv}$ to $e'_{uv}$. Then, we add an edge between $w_{uv}$ and $w'_{uv}$, as shown in Figure \ref{fig4}. In this way, we define the graph $H_i$ with the following sets of vertices and edges:
\[
V(H_i)=V(H_i')\cup \lbrace w_{uv}, w'_{uv}: \{u, v\} \in R_i\rbrace;
\]
\[
E(H_i)=\lbrace uv\in E(H_i'): \{u,v\} \in S_i\rbrace\cup \lbrace w_{uv}u, w_{uv}v, w'_{uv}u, w'_{uv}v, w_{uv}w'_{uv}: \{u,v\} \in R_i\rbrace,
\]
where $S_i$ is the set of pairs of vertices $u,v\in V(H_i')$ that do not define parallel edges in $H_i'$; and $R_i$ is the set of pairs of vertices $u,v\in V(H_i')$ that define parallel edges in $H_i'$.

\begin{figure}[ht]
    \centering

\tikzset{every picture/.style={line width=0.75pt}} 

\begin{tikzpicture}[x=0.65pt,y=0.65pt,yscale=-1,xscale=1]

\draw    (112.53,82.17) -- (196.03,83.17) ;
\draw [shift={(196.03,83.17)}, rotate = 0.69] [color={rgb, 255:red, 0; green, 0; blue, 0 }  ][fill={rgb, 255:red, 0; green, 0; blue, 0 }  ][line width=0.75]      (0, 0) circle [x radius= 3.35, y radius= 3.35]   ;
\draw [shift={(112.53,82.17)}, rotate = 0.69] [color={rgb, 255:red, 0; green, 0; blue, 0 }  ][fill={rgb, 255:red, 0; green, 0; blue, 0 }  ][line width=0.75]      (0, 0) circle [x radius= 3.35, y radius= 3.35]   ;
\draw [color={rgb, 255:red, 208; green, 2; blue, 27 }  ,draw opacity=1 ]   (112.53,82.17) .. controls (128.03,38.17) and (191.03,50.17) .. (196.03,83.17) ;
\draw    (292.87,83.83) -- (334.62,84.33) ;
\draw [shift={(334.62,84.33)}, rotate = 0.69] [color={rgb, 255:red, 0; green, 0; blue, 0 }  ][fill={rgb, 255:red, 0; green, 0; blue, 0 }  ][line width=0.75]      (0, 0) circle [x radius= 3.35, y radius= 3.35]   ;
\draw [shift={(292.87,83.83)}, rotate = 0.69] [color={rgb, 255:red, 0; green, 0; blue, 0 }  ][fill={rgb, 255:red, 0; green, 0; blue, 0 }  ][line width=0.75]      (0, 0) circle [x radius= 3.35, y radius= 3.35]   ;
\draw    (334.62,84.33) -- (376.37,84.83) ;
\draw [shift={(376.37,84.83)}, rotate = 0.69] [color={rgb, 255:red, 0; green, 0; blue, 0 }  ][fill={rgb, 255:red, 0; green, 0; blue, 0 }  ][line width=0.75]      (0, 0) circle [x radius= 3.35, y radius= 3.35]   ;
\draw [shift={(334.62,84.33)}, rotate = 0.69] [color={rgb, 255:red, 0; green, 0; blue, 0 }  ][fill={rgb, 255:red, 0; green, 0; blue, 0 }  ][line width=0.75]      (0, 0) circle [x radius= 3.35, y radius= 3.35]   ;
\draw [color={rgb, 255:red, 208; green, 2; blue, 27 }  ,draw opacity=1 ]   (292.87,83.83) .. controls (308.37,39.83) and (371.37,51.83) .. (376.37,84.83) ;
\draw [shift={(335.25,55.27)}, rotate = 4.76] [color={rgb, 255:red, 208; green, 2; blue, 27 }  ,draw opacity=1 ][fill={rgb, 255:red, 208; green, 2; blue, 27 }  ,fill opacity=1 ][line width=0.75]      (0, 0) circle [x radius= 3.35, y radius= 3.35]   ;
\draw    (459.53,86.5) -- (501.28,87) ;
\draw [shift={(501.28,87)}, rotate = 0.69] [color={rgb, 255:red, 0; green, 0; blue, 0 }  ][fill={rgb, 255:red, 0; green, 0; blue, 0 }  ][line width=0.75]      (0, 0) circle [x radius= 3.35, y radius= 3.35]   ;
\draw [shift={(459.53,86.5)}, rotate = 0.69] [color={rgb, 255:red, 0; green, 0; blue, 0 }  ][fill={rgb, 255:red, 0; green, 0; blue, 0 }  ][line width=0.75]      (0, 0) circle [x radius= 3.35, y radius= 3.35]   ;
\draw    (501.28,87) -- (543.03,87.5) ;
\draw [shift={(543.03,87.5)}, rotate = 0.69] [color={rgb, 255:red, 0; green, 0; blue, 0 }  ][fill={rgb, 255:red, 0; green, 0; blue, 0 }  ][line width=0.75]      (0, 0) circle [x radius= 3.35, y radius= 3.35]   ;
\draw [shift={(501.28,87)}, rotate = 0.69] [color={rgb, 255:red, 0; green, 0; blue, 0 }  ][fill={rgb, 255:red, 0; green, 0; blue, 0 }  ][line width=0.75]      (0, 0) circle [x radius= 3.35, y radius= 3.35]   ;
\draw [color={rgb, 255:red, 208; green, 2; blue, 27 }  ,draw opacity=1 ]   (459.53,86.5) .. controls (475.03,42.5) and (538.03,54.5) .. (543.03,87.5) ;
\draw [shift={(501.91,57.94)}, rotate = 4.76] [color={rgb, 255:red, 208; green, 2; blue, 27 }  ,draw opacity=1 ][fill={rgb, 255:red, 208; green, 2; blue, 27 }  ,fill opacity=1 ][line width=0.75]      (0, 0) circle [x radius= 3.35, y radius= 3.35]   ;
\draw    (501.28,87) -- (502.2,58.17) ;
\draw [shift={(502.2,58.17)}, rotate = 271.82] [color={rgb, 255:red, 0; green, 0; blue, 0 }  ][fill={rgb, 255:red, 0; green, 0; blue, 0 }  ][line width=0.75]      (0, 0) circle [x radius= 3.35, y radius= 3.35]   ;
\draw   (227.5,69.38) -- (242.8,69.38) -- (242.8,65) -- (253,73.75) -- (242.8,82.5) -- (242.8,78.13) -- (227.5,78.13) -- cycle ;
\draw   (403,69.38) -- (418.3,69.38) -- (418.3,65) -- (428.5,73.75) -- (418.3,82.5) -- (418.3,78.13) -- (403,78.13) -- cycle ;

\draw (91.53,71.07) node [anchor=north west][inner sep=0.75pt]    {$u$};
\draw (203.53,73.07) node [anchor=north west][inner sep=0.75pt]    {$v$};
\draw (191.7,17.23) node [anchor=north west][inner sep=0.75pt]    {$H_{i} '$};
\draw (146.2,80.4) node [anchor=north west][inner sep=0.75pt]    {$e_{uv}$};
\draw (143.7,31.4) node [anchor=north west][inner sep=0.75pt]    {$e'_{uv}$};
\draw (271.87,72.73) node [anchor=north west][inner sep=0.75pt]    {$u$};
\draw (383.87,74.73) node [anchor=north west][inner sep=0.75pt]    {$v$};
\draw (322.53,93.9) node [anchor=north west][inner sep=0.75pt]    {$w_{uv}$};
\draw (321.2,27.9) node [anchor=north west][inner sep=0.75pt]    {$w'_{uv}$};
\draw (438.53,75.4) node [anchor=north west][inner sep=0.75pt]    {$u$};
\draw (550.53,77.4) node [anchor=north west][inner sep=0.75pt]    {$v$};
\draw (489.2,96.57) node [anchor=north west][inner sep=0.75pt]    {$w_{uv}$};
\draw (487.87,30.57) node [anchor=north west][inner sep=0.75pt]    {$w'_{uv}$};
\draw (553.4,23.5) node [anchor=north west][inner sep=0.75pt]    {$H_{i}$};

\end{tikzpicture}
    \caption{Construction of the $H_i$}
    \label{fig4}
\end{figure}
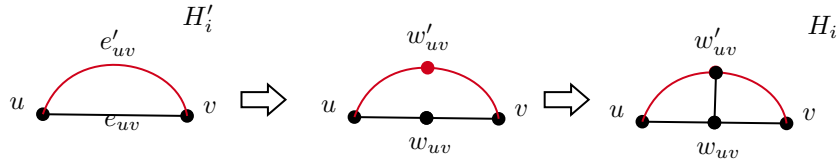

Thus, $H_i$ is a finite, bridgeless, and cubic graph, since $G_{m_i}$ is as well. Furthermore, given a fixed non-principal ultrafilter $F$ over $\mathbb{N}$, $G$ is the $F$-limit of the sequence $\langle H_i\rangle_{i\in\mathbb{N}}$. Indeed, given $v \in V(G)$, there exists $n \in \mathbb{N}$ such that $v \in V(G_m)$ and $N_G(v) \subset V(G_m)$ for all $m \geq n$, since $G = \bigcup_{n\in\mathbb{N}} G_n$. Consequently, $v \in V(H'_m)$ for sufficiently large $m$, as it is not an internal vertex of any path to be contracted. Since $F$ is a non-principal ultrafilter (and thus contains all cofinite sets of $\mathbb{N}$) and $\{ m \in \mathbb{N} : v \in V(G_m) \} = \{ m \in \mathbb{N} : m \geq n \}$, it follows that $\{ m \in \mathbb{N} : v \in V(G_m) \} \in F$.

Given an edge $e = uv \in E(G)$, there exists $n \in \mathbb{N}$ such that $e \in E(G_m)$ and $N_G(u), N_G(v),$ $N_G(N_G(v)), N_G(N_G(u)) \subset V(G_m)$ for all $m \geq n$. Thus, the vertices $u, v$ and their neighbors have degree 3 in $G_m$ for all $m \geq n$. Therefore, $u, v \in V(H'_m)$ and $e = uv \in E(H'_m)$; moreover, the edge $e$ is not a parallel edge in $H_m'$. Thus, $e \in E(H_m)$ for all $m \geq n$. As before, the set of such indices belongs to $F$. Hence, $G$ is contained in the $F$-limit of $\langle H_i\rangle_{i\in\mathbb{N}}$.

We now show that the only vertices and edges of the $F$-limit are those of $G$. Let $u, v \in V(G)$. Note that there exists $n \in \mathbb{N}$ such that for all $m \geq n$, $u$ and $v$ are not endpoints of paths in $\mathcal{P}_m$ that would yield vertices $w_{uv}$ or $w'_{uv}$. Thus, $w_{uv}, w'_{uv} \notin V(H_m)$ for all sufficiently large $m$. Therefore, these vertices are not in the $F$-limit. Analogously, it can be proven that the edges of $H_i$ created from path contractions or the edges $w_{uv}w'_{uv}$ also do not belong to the $F$-limit. Consequently, the $F$-limit of $\langle H_i\rangle_{i\in\mathbb{N}}$ is exactly the graph $G$. This completes the proof of the lemma.
\end{proof}

\begin{lemma}\label{lemma2}
    Let $G$ be an infinite, cubic, bridgeless, connected graph, and let $\langle H_i \rangle_{i \in \mathbb{N}}$ be a sequence of finite graphs as in the proof of Lemma~\ref{thm1}. For each $H_i$, suppose that $H_i$ has a perfect matching $M_i$. If $M \subset E(G)$ is the $F$-limit of the sequence of perfect matchings $\langle M_i \rangle_{i \in \mathbb{N}}$, then $M$ is a perfect matching of $G$.
\end{lemma}
\begin{proof}
    First, we will show that all vertices of $G$ are covered by $M$. Let $F$ be a fixed non-principal ultrafilter over $\mathbb{N}$. Given $v \in V(G)$, the neighborhood $N_G(v)$ is finite since $G$ is locally finite. Since $G$ is the $F$-limit of $\langle H_n\rangle_{n\in\mathbb{N}}$, we have $\{ n \in \mathbb{N} : v \in V(H_n) \} \in F$. Furthermore, by the construction of the graphs $H_i$, if a vertex $v \in V(G_i)$ satisfies $N_{G_i}(v) = N_G(v)$, then $N_{G_{m_i}}(v) = N_G(v)$ and $N_{G_{m_i}}(w) = N_G(w)$ for all $w \in N_G(v)$. Thus, $N_{H_i}(v) = N_G(v)$. Consequently, there exists $n \in \mathbb{N}$ such that $N_{H_i}(v) = N_G(v)$ for all $i \geq n$.

Since $M_i$ is a perfect matching for $H_i$, for each $i \geq n$ there exists a neighbor $w \in N_G(v)$ such that $vw \in M_i$. Let $E_k = \{ i \in \mathbb{N} : i \geq n \text{ and } vw_k \in M_i \}$, where $N_G(v) = \{ w_1, w_2, w_3 \}$. Since each $M_i$ is a perfect matching, it follows that $E_k \cap E_l = \emptyset$ for any distinct $k, l \in \{ 1, 2, 3 \}$. Moreover,
\[
E_1 \cup E_2 \cup E_3 = \{ i \in \mathbb{N} : v \in V(G_i) \} = \{ i \in \mathbb{N} : i \geq n \} \in F.
\]
Thus, since $F$ is a non-principal ultrafilter, there exists a unique $k \in \{ 1, 2, 3 \}$ such that $E_k \in F$. Therefore, $vw_k \in M$. Furthermore, since this $E_k$ is unique, we cannot have $vw_l \in M$ for $l \neq k$. Hence, $M$ contains exactly one edge incident to $v$. Since $v$ is an arbitrary vertex of $G$, it follows that $M$ is a perfect matching for $G$.
\end{proof}

Let $k,p\in\mathbb{N}_{>0}$ with $p\leq k$, and let $G$ be a graph. A \textbf{$(k,p)$-perfect matching cover} of $G$ (\textbf{$(k,p)$-PM cover}) is a collection of $k$ perfect matchings such that each edge of $G$ belongs to at most $p$ of them. A \textbf{$(k,p)^*$-perfect matching cover} of $G$ (\textbf{$(k,p)^*$-PM cover}) is a collection of $k$ perfect matchings such that each edge of $G$ belongs to at least $p$ of them. Observe that a collection of perfect matchings $\mathcal{M}$ of a graph $G$ is both a $(k,p)$-PM cover and a $(k,p)^*$-PM cover if and only if each edge of $G$ belongs to exactly $p$ perfect matchings in the collection $\mathcal{M}$.

\begin{lemma}\label{lemma3}
    Let $G$ be a finite, cubic, bridgeless, connected graph. There exists an infinite, cubic, bridgeless, connected graph $H$ such that if $H$ has a $(k,p)$-PM cover ($(k,p)^*$-PM cover) for some $k,p \in \mathbb{N}_{>0}$, then so does $G$.
\end{lemma}
\begin{proof}
    Fix $u, v \in V(G)$ with $u \neq v$ and $uv \in E(G)$. We will construct an infinite, cubic, and bridgeless graph $H$ from the graph $G$. Consider $G' = G \setminus \{uv\}$ and an infinite family $\{G_i\}_{i\in\mathbb{Z}}$ of pairwise disjoint copies of $G'$. Let us denote by $u_i$ and $v_i$ the vertices of $G_i$ that are copies of the vertices $u$ and $v$, respectively. Note that, for each $i\in\mathbb{Z}$, the vertices of $G_i$ have degree 3 in $G_i$, except for the vertices $u_i$ and $v_i$, which have degree 2 in $G_i$. Let $L$ be the infinite double ladder, as shown in Figure \ref{fig6}, that is, $L$ is the graph composed of:
   
    $$V(L)=\lbrace x_i, y_i: i\in\mathbb{Z}\rbrace\quad  ; \quad E(L)=\lbrace x_ix_{i+1}: i\in\mathbb{Z}\rbrace \cup \lbrace y_iy_{i+1}: i\in\mathbb{Z}\rbrace\cup \lbrace x_iy_{i}: i\in\mathbb{Z}\rbrace.$$  
    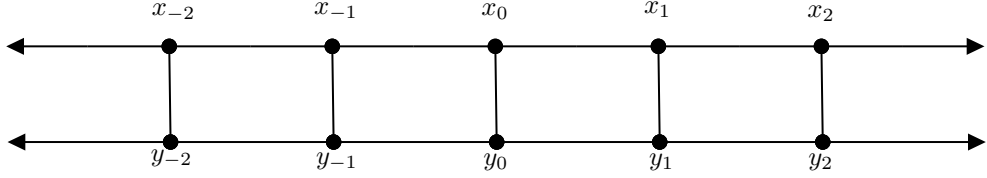
\begin{figure}[ht]
        \centering

\tikzset{every picture/.style={line width=0.75pt}} 

\begin{tikzpicture}[x=0.75pt,y=0.75pt,yscale=-1,xscale=1]

\draw    (89,36.07) -- (126.9,36.07) ;
\draw [shift={(86,36.07)}, rotate = 0] [fill={rgb, 255:red, 0; green, 0; blue, 0 }  ][line width=0.08]  [draw opacity=0] (8.93,-4.29) -- (0,0) -- (8.93,4.29) -- cycle    ;
\draw    (126.9,36.07) -- (167.8,36.07) ;
\draw [shift={(167.8,36.07)}, rotate = 0] [color={rgb, 255:red, 0; green, 0; blue, 0 }  ][fill={rgb, 255:red, 0; green, 0; blue, 0 }  ][line width=0.75]      (0, 0) circle [x radius= 3.35, y radius= 3.35]   ;
\draw    (167.8,36.07) -- (208.69,36.07) ;
\draw [shift={(167.8,36.07)}, rotate = 360] [color={rgb, 255:red, 0; green, 0; blue, 0 }  ][fill={rgb, 255:red, 0; green, 0; blue, 0 }  ][line width=0.75]      (0, 0) circle [x radius= 3.35, y radius= 3.35]   ;
\draw    (208.69,36.07) -- (249.59,36.07) ;
\draw [shift={(249.59,36.07)}, rotate = 0] [color={rgb, 255:red, 0; green, 0; blue, 0 }  ][fill={rgb, 255:red, 0; green, 0; blue, 0 }  ][line width=0.75]      (0, 0) circle [x radius= 3.35, y radius= 3.35]   ;
\draw    (249.59,36.07) -- (290.49,36.07) ;
\draw [shift={(249.59,36.07)}, rotate = 0] [color={rgb, 255:red, 0; green, 0; blue, 0 }  ][fill={rgb, 255:red, 0; green, 0; blue, 0 }  ][line width=0.75]      (0, 0) circle [x radius= 3.35, y radius= 3.35]   ;
\draw    (290.49,36.07) -- (331.39,36.07) ;
\draw [shift={(331.39,36.07)}, rotate = 0] [color={rgb, 255:red, 0; green, 0; blue, 0 }  ][fill={rgb, 255:red, 0; green, 0; blue, 0 }  ][line width=0.75]      (0, 0) circle [x radius= 3.35, y radius= 3.35]   ;
\draw    (331.39,36.07) -- (372.29,36.07) ;
\draw    (372.29,36.07) -- (413.19,36.07) ;
\draw [shift={(413.19,36.07)}, rotate = 360] [color={rgb, 255:red, 0; green, 0; blue, 0 }  ][fill={rgb, 255:red, 0; green, 0; blue, 0 }  ][line width=0.75]      (0, 0) circle [x radius= 3.35, y radius= 3.35]   ;
\draw    (413.19,36.07) -- (435.5,36.07) -- (454.08,36.07) ;
\draw [shift={(413.19,36.07)}, rotate = 0] [color={rgb, 255:red, 0; green, 0; blue, 0 }  ][fill={rgb, 255:red, 0; green, 0; blue, 0 }  ][line width=0.75]      (0, 0) circle [x radius= 3.35, y radius= 3.35]   ;
\draw    (454.08,36.07) -- (494.98,36.07) ;
\draw [shift={(494.98,36.07)}, rotate = 360] [color={rgb, 255:red, 0; green, 0; blue, 0 }  ][fill={rgb, 255:red, 0; green, 0; blue, 0 }  ][line width=0.75]      (0, 0) circle [x radius= 3.35, y radius= 3.35]   ;
\draw    (494.98,36.07) -- (535.88,36.07) ;
\draw [shift={(494.98,36.07)}, rotate = 0] [color={rgb, 255:red, 0; green, 0; blue, 0 }  ][fill={rgb, 255:red, 0; green, 0; blue, 0 }  ][line width=0.75]      (0, 0) circle [x radius= 3.35, y radius= 3.35]   ;
\draw    (535.88,36.07) -- (573.78,36.07) ;
\draw [shift={(576.78,36.07)}, rotate = 180] [fill={rgb, 255:red, 0; green, 0; blue, 0 }  ][line width=0.08]  [draw opacity=0] (8.93,-4.29) -- (0,0) -- (8.93,4.29) -- cycle    ;
\draw    (89.67,84.07) -- (168.46,84.07) ;
\draw [shift={(168.46,84.07)}, rotate = 0] [color={rgb, 255:red, 0; green, 0; blue, 0 }  ][fill={rgb, 255:red, 0; green, 0; blue, 0 }  ][line width=0.75]      (0, 0) circle [x radius= 3.35, y radius= 3.35]   ;
\draw [shift={(86.67,84.07)}, rotate = 0] [fill={rgb, 255:red, 0; green, 0; blue, 0 }  ][line width=0.08]  [draw opacity=0] (8.93,-4.29) -- (0,0) -- (8.93,4.29) -- cycle    ;
\draw    (168.46,84.07) -- (167.8,36.07) ;
\draw [shift={(167.8,36.07)}, rotate = 269.2] [color={rgb, 255:red, 0; green, 0; blue, 0 }  ][fill={rgb, 255:red, 0; green, 0; blue, 0 }  ][line width=0.75]      (0, 0) circle [x radius= 3.35, y radius= 3.35]   ;
\draw [shift={(168.46,84.07)}, rotate = 269.2] [color={rgb, 255:red, 0; green, 0; blue, 0 }  ][fill={rgb, 255:red, 0; green, 0; blue, 0 }  ][line width=0.75]      (0, 0) circle [x radius= 3.35, y radius= 3.35]   ;
\draw    (168.46,84.07) -- (250.26,84.07) ;
\draw [shift={(250.26,84.07)}, rotate = 0] [color={rgb, 255:red, 0; green, 0; blue, 0 }  ][fill={rgb, 255:red, 0; green, 0; blue, 0 }  ][line width=0.75]      (0, 0) circle [x radius= 3.35, y radius= 3.35]   ;
\draw [shift={(168.46,84.07)}, rotate = 0] [color={rgb, 255:red, 0; green, 0; blue, 0 }  ][fill={rgb, 255:red, 0; green, 0; blue, 0 }  ][line width=0.75]      (0, 0) circle [x radius= 3.35, y radius= 3.35]   ;
\draw    (250.26,84.07) -- (249.59,36.07) ;
\draw [shift={(249.59,36.07)}, rotate = 269.2] [color={rgb, 255:red, 0; green, 0; blue, 0 }  ][fill={rgb, 255:red, 0; green, 0; blue, 0 }  ][line width=0.75]      (0, 0) circle [x radius= 3.35, y radius= 3.35]   ;
\draw [shift={(250.26,84.07)}, rotate = 269.2] [color={rgb, 255:red, 0; green, 0; blue, 0 }  ][fill={rgb, 255:red, 0; green, 0; blue, 0 }  ][line width=0.75]      (0, 0) circle [x radius= 3.35, y radius= 3.35]   ;
\draw    (250.26,84.07) -- (332.06,84.07) ;
\draw [shift={(332.06,84.07)}, rotate = 0] [color={rgb, 255:red, 0; green, 0; blue, 0 }  ][fill={rgb, 255:red, 0; green, 0; blue, 0 }  ][line width=0.75]      (0, 0) circle [x radius= 3.35, y radius= 3.35]   ;
\draw [shift={(250.26,84.07)}, rotate = 0] [color={rgb, 255:red, 0; green, 0; blue, 0 }  ][fill={rgb, 255:red, 0; green, 0; blue, 0 }  ][line width=0.75]      (0, 0) circle [x radius= 3.35, y radius= 3.35]   ;
\draw    (332.06,84.07) -- (331.39,36.07) ;
\draw [shift={(331.39,36.07)}, rotate = 269.2] [color={rgb, 255:red, 0; green, 0; blue, 0 }  ][fill={rgb, 255:red, 0; green, 0; blue, 0 }  ][line width=0.75]      (0, 0) circle [x radius= 3.35, y radius= 3.35]   ;
\draw [shift={(332.06,84.07)}, rotate = 269.2] [color={rgb, 255:red, 0; green, 0; blue, 0 }  ][fill={rgb, 255:red, 0; green, 0; blue, 0 }  ][line width=0.75]      (0, 0) circle [x radius= 3.35, y radius= 3.35]   ;
\draw    (332.06,84.07) -- (413.85,84.07) ;
\draw [shift={(413.85,84.07)}, rotate = 360] [color={rgb, 255:red, 0; green, 0; blue, 0 }  ][fill={rgb, 255:red, 0; green, 0; blue, 0 }  ][line width=0.75]      (0, 0) circle [x radius= 3.35, y radius= 3.35]   ;
\draw [shift={(332.06,84.07)}, rotate = 360] [color={rgb, 255:red, 0; green, 0; blue, 0 }  ][fill={rgb, 255:red, 0; green, 0; blue, 0 }  ][line width=0.75]      (0, 0) circle [x radius= 3.35, y radius= 3.35]   ;
\draw    (413.85,84.07) -- (413.19,36.07) ;
\draw [shift={(413.19,36.07)}, rotate = 269.2] [color={rgb, 255:red, 0; green, 0; blue, 0 }  ][fill={rgb, 255:red, 0; green, 0; blue, 0 }  ][line width=0.75]      (0, 0) circle [x radius= 3.35, y radius= 3.35]   ;
\draw [shift={(413.85,84.07)}, rotate = 269.2] [color={rgb, 255:red, 0; green, 0; blue, 0 }  ][fill={rgb, 255:red, 0; green, 0; blue, 0 }  ][line width=0.75]      (0, 0) circle [x radius= 3.35, y radius= 3.35]   ;
\draw    (413.85,84.07) -- (495.65,84.07) ;
\draw [shift={(495.65,84.07)}, rotate = 0] [color={rgb, 255:red, 0; green, 0; blue, 0 }  ][fill={rgb, 255:red, 0; green, 0; blue, 0 }  ][line width=0.75]      (0, 0) circle [x radius= 3.35, y radius= 3.35]   ;
\draw [shift={(413.85,84.07)}, rotate = 0] [color={rgb, 255:red, 0; green, 0; blue, 0 }  ][fill={rgb, 255:red, 0; green, 0; blue, 0 }  ][line width=0.75]      (0, 0) circle [x radius= 3.35, y radius= 3.35]   ;
\draw    (494.98,36.07) -- (495.65,84.07) ;
\draw [shift={(495.65,84.07)}, rotate = 89.2] [color={rgb, 255:red, 0; green, 0; blue, 0 }  ][fill={rgb, 255:red, 0; green, 0; blue, 0 }  ][line width=0.75]      (0, 0) circle [x radius= 3.35, y radius= 3.35]   ;
\draw [shift={(494.98,36.07)}, rotate = 89.2] [color={rgb, 255:red, 0; green, 0; blue, 0 }  ][fill={rgb, 255:red, 0; green, 0; blue, 0 }  ][line width=0.75]      (0, 0) circle [x radius= 3.35, y radius= 3.35]   ;
\draw    (495.65,84.07) -- (574.45,84.07) ;
\draw [shift={(577.45,84.07)}, rotate = 180] [fill={rgb, 255:red, 0; green, 0; blue, 0 }  ][line width=0.08]  [draw opacity=0] (8.93,-4.29) -- (0,0) -- (8.93,4.29) -- cycle    ;
\draw [shift={(495.65,84.07)}, rotate = 0] [color={rgb, 255:red, 0; green, 0; blue, 0 }  ][fill={rgb, 255:red, 0; green, 0; blue, 0 }  ][line width=0.75]      (0, 0) circle [x radius= 3.35, y radius= 3.35]   ;

\draw (323,12.9) node [anchor=north west][inner sep=0.75pt]    {$x_{0}$};
\draw (404.5,11.9) node [anchor=north west][inner sep=0.75pt]    {$x_{1}$};
\draw (486.5,13.9) node [anchor=north west][inner sep=0.75pt]    {$x_{2}$};
\draw (239,12.4) node [anchor=north west][inner sep=0.75pt]    {$x_{-1}$};
\draw (157.5,12.9) node [anchor=north west][inner sep=0.75pt]    {$x_{-2}$};
\draw (324,88.4) node [anchor=north west][inner sep=0.75pt]    {$y_{0}$};
\draw (407,88.4) node [anchor=north west][inner sep=0.75pt]    {$y_{1}$};
\draw (487,88.4) node [anchor=north west][inner sep=0.75pt]    {$y_{2}$};
\draw (240,87.9) node [anchor=north west][inner sep=0.75pt]    {$y_{-1}$};
\draw (157,86.4) node [anchor=north west][inner sep=0.75pt]    {$y_{-2}$};

\end{tikzpicture}
        \caption{Double ladder graph}
        \label{fig6}
    \end{figure}
    Consider $L'$ the graph obtained from $L$ by subdividing each edge $x_ix_{i+1}$, that is, for each edge $x_ix_{i+1}\in E(L)$, we add a vertex $w_{i+1}$. Now, we define $H$ as follows (see Figure~\ref{fig7}): $V(H)=V(L')\cup \bigcup_{i\in\mathbb{Z}} V(G_i)$, recalling that each $G_i$ is a copy of $G'$. The edge set is determined as follows:

    $$E(H)= E(L')\cup \left(\bigcup_{i\in\mathbb{Z}} E(G_i)\right)\cup \lbrace u_iw_{2i}, v_iw_{2i+1}: i\in\mathbb{Z}\rbrace.$$

    \begin{figure}[ht]
        \centering

\tikzset{every picture/.style={line width=0.75pt}} 

\begin{tikzpicture}[x=0.75pt,y=0.75pt,yscale=-1,xscale=1]

\draw   (115.56,43.93) .. controls (115.56,22.43) and (133.88,5) .. (156.46,5) .. controls (179.05,5) and (197.36,22.43) .. (197.36,43.93) .. controls (197.36,65.43) and (179.05,82.86) .. (156.46,82.86) .. controls (133.88,82.86) and (115.56,65.43) .. (115.56,43.93) -- cycle ;
\draw    (115.56,44.74) -- (115.9,125.07) ;
\draw [shift={(115.9,125.07)}, rotate = 89.76] [color={rgb, 255:red, 0; green, 0; blue, 0 }  ][fill={rgb, 255:red, 0; green, 0; blue, 0 }  ][line width=0.75]      (0, 0) circle [x radius= 3.35, y radius= 3.35]   ;
\draw [shift={(115.56,44.74)}, rotate = 89.76] [color={rgb, 255:red, 0; green, 0; blue, 0 }  ][fill={rgb, 255:red, 0; green, 0; blue, 0 }  ][line width=0.75]      (0, 0) circle [x radius= 3.35, y radius= 3.35]   ;
\draw  [dash pattern={on 0.84pt off 2.51pt}]  (72.33,101.12) -- (94,102) ;
\draw  [dash pattern={on 0.84pt off 2.51pt}]  (552,93) -- (571.33,93.46) ;
\draw    (78,125.07) -- (115.9,125.07) ;
\draw [shift={(115.9,125.07)}, rotate = 0] [color={rgb, 255:red, 0; green, 0; blue, 0 }  ][fill={rgb, 255:red, 0; green, 0; blue, 0 }  ][line width=0.75]      (0, 0) circle [x radius= 3.35, y radius= 3.35]   ;
\draw [shift={(75,125.07)}, rotate = 0] [fill={rgb, 255:red, 0; green, 0; blue, 0 }  ][line width=0.08]  [draw opacity=0] (8.93,-4.29) -- (0,0) -- (8.93,4.29) -- cycle    ;
\draw    (115.9,125.07) -- (156.8,125.07) ;
\draw [shift={(156.8,125.07)}, rotate = 0] [color={rgb, 255:red, 0; green, 0; blue, 0 }  ][fill={rgb, 255:red, 0; green, 0; blue, 0 }  ][line width=0.75]      (0, 0) circle [x radius= 3.35, y radius= 3.35]   ;
\draw [shift={(115.9,125.07)}, rotate = 0] [color={rgb, 255:red, 0; green, 0; blue, 0 }  ][fill={rgb, 255:red, 0; green, 0; blue, 0 }  ][line width=0.75]      (0, 0) circle [x radius= 3.35, y radius= 3.35]   ;
\draw    (156.8,125.07) -- (197.69,125.07) ;
\draw [shift={(197.69,125.07)}, rotate = 360] [color={rgb, 255:red, 0; green, 0; blue, 0 }  ][fill={rgb, 255:red, 0; green, 0; blue, 0 }  ][line width=0.75]      (0, 0) circle [x radius= 3.35, y radius= 3.35]   ;
\draw [shift={(156.8,125.07)}, rotate = 360] [color={rgb, 255:red, 0; green, 0; blue, 0 }  ][fill={rgb, 255:red, 0; green, 0; blue, 0 }  ][line width=0.75]      (0, 0) circle [x radius= 3.35, y radius= 3.35]   ;
\draw    (197.69,125.07) -- (238.59,125.07) ;
\draw [shift={(238.59,125.07)}, rotate = 0] [color={rgb, 255:red, 0; green, 0; blue, 0 }  ][fill={rgb, 255:red, 0; green, 0; blue, 0 }  ][line width=0.75]      (0, 0) circle [x radius= 3.35, y radius= 3.35]   ;
\draw [shift={(197.69,125.07)}, rotate = 0] [color={rgb, 255:red, 0; green, 0; blue, 0 }  ][fill={rgb, 255:red, 0; green, 0; blue, 0 }  ][line width=0.75]      (0, 0) circle [x radius= 3.35, y radius= 3.35]   ;
\draw    (238.59,125.07) -- (279.49,125.07) ;
\draw [shift={(279.49,125.07)}, rotate = 0] [color={rgb, 255:red, 0; green, 0; blue, 0 }  ][fill={rgb, 255:red, 0; green, 0; blue, 0 }  ][line width=0.75]      (0, 0) circle [x radius= 3.35, y radius= 3.35]   ;
\draw [shift={(238.59,125.07)}, rotate = 0] [color={rgb, 255:red, 0; green, 0; blue, 0 }  ][fill={rgb, 255:red, 0; green, 0; blue, 0 }  ][line width=0.75]      (0, 0) circle [x radius= 3.35, y radius= 3.35]   ;
\draw    (279.49,125.07) -- (320.39,125.07) ;
\draw [shift={(320.39,125.07)}, rotate = 0] [color={rgb, 255:red, 0; green, 0; blue, 0 }  ][fill={rgb, 255:red, 0; green, 0; blue, 0 }  ][line width=0.75]      (0, 0) circle [x radius= 3.35, y radius= 3.35]   ;
\draw [shift={(279.49,125.07)}, rotate = 0] [color={rgb, 255:red, 0; green, 0; blue, 0 }  ][fill={rgb, 255:red, 0; green, 0; blue, 0 }  ][line width=0.75]      (0, 0) circle [x radius= 3.35, y radius= 3.35]   ;
\draw    (320.39,125.07) -- (361.29,125.07) ;
\draw [shift={(361.29,125.07)}, rotate = 0] [color={rgb, 255:red, 0; green, 0; blue, 0 }  ][fill={rgb, 255:red, 0; green, 0; blue, 0 }  ][line width=0.75]      (0, 0) circle [x radius= 3.35, y radius= 3.35]   ;
\draw [shift={(320.39,125.07)}, rotate = 0] [color={rgb, 255:red, 0; green, 0; blue, 0 }  ][fill={rgb, 255:red, 0; green, 0; blue, 0 }  ][line width=0.75]      (0, 0) circle [x radius= 3.35, y radius= 3.35]   ;
\draw    (361.29,125.07) -- (402.19,125.07) ;
\draw [shift={(402.19,125.07)}, rotate = 360] [color={rgb, 255:red, 0; green, 0; blue, 0 }  ][fill={rgb, 255:red, 0; green, 0; blue, 0 }  ][line width=0.75]      (0, 0) circle [x radius= 3.35, y radius= 3.35]   ;
\draw [shift={(361.29,125.07)}, rotate = 360] [color={rgb, 255:red, 0; green, 0; blue, 0 }  ][fill={rgb, 255:red, 0; green, 0; blue, 0 }  ][line width=0.75]      (0, 0) circle [x radius= 3.35, y radius= 3.35]   ;
\draw    (402.19,125.07) -- (443.08,125.07) ;
\draw [shift={(443.08,125.07)}, rotate = 0] [color={rgb, 255:red, 0; green, 0; blue, 0 }  ][fill={rgb, 255:red, 0; green, 0; blue, 0 }  ][line width=0.75]      (0, 0) circle [x radius= 3.35, y radius= 3.35]   ;
\draw [shift={(402.19,125.07)}, rotate = 0] [color={rgb, 255:red, 0; green, 0; blue, 0 }  ][fill={rgb, 255:red, 0; green, 0; blue, 0 }  ][line width=0.75]      (0, 0) circle [x radius= 3.35, y radius= 3.35]   ;
\draw    (443.08,125.07) -- (483.98,125.07) ;
\draw [shift={(483.98,125.07)}, rotate = 360] [color={rgb, 255:red, 0; green, 0; blue, 0 }  ][fill={rgb, 255:red, 0; green, 0; blue, 0 }  ][line width=0.75]      (0, 0) circle [x radius= 3.35, y radius= 3.35]   ;
\draw [shift={(443.08,125.07)}, rotate = 360] [color={rgb, 255:red, 0; green, 0; blue, 0 }  ][fill={rgb, 255:red, 0; green, 0; blue, 0 }  ][line width=0.75]      (0, 0) circle [x radius= 3.35, y radius= 3.35]   ;
\draw    (483.98,125.07) -- (524.88,125.07) ;
\draw [shift={(524.88,125.07)}, rotate = 0] [color={rgb, 255:red, 0; green, 0; blue, 0 }  ][fill={rgb, 255:red, 0; green, 0; blue, 0 }  ][line width=0.75]      (0, 0) circle [x radius= 3.35, y radius= 3.35]   ;
\draw [shift={(483.98,125.07)}, rotate = 0] [color={rgb, 255:red, 0; green, 0; blue, 0 }  ][fill={rgb, 255:red, 0; green, 0; blue, 0 }  ][line width=0.75]      (0, 0) circle [x radius= 3.35, y radius= 3.35]   ;
\draw    (524.88,125.07) -- (562.78,125.07) ;
\draw [shift={(565.78,125.07)}, rotate = 180] [fill={rgb, 255:red, 0; green, 0; blue, 0 }  ][line width=0.08]  [draw opacity=0] (8.93,-4.29) -- (0,0) -- (8.93,4.29) -- cycle    ;
\draw [shift={(524.88,125.07)}, rotate = 0] [color={rgb, 255:red, 0; green, 0; blue, 0 }  ][fill={rgb, 255:red, 0; green, 0; blue, 0 }  ][line width=0.75]      (0, 0) circle [x radius= 3.35, y radius= 3.35]   ;
\draw    (78.67,173.07) -- (157.46,173.07) ;
\draw [shift={(157.46,173.07)}, rotate = 0] [color={rgb, 255:red, 0; green, 0; blue, 0 }  ][fill={rgb, 255:red, 0; green, 0; blue, 0 }  ][line width=0.75]      (0, 0) circle [x radius= 3.35, y radius= 3.35]   ;
\draw [shift={(75.67,173.07)}, rotate = 0] [fill={rgb, 255:red, 0; green, 0; blue, 0 }  ][line width=0.08]  [draw opacity=0] (8.93,-4.29) -- (0,0) -- (8.93,4.29) -- cycle    ;
\draw    (157.46,173.07) -- (156.8,125.07) ;
\draw [shift={(156.8,125.07)}, rotate = 269.2] [color={rgb, 255:red, 0; green, 0; blue, 0 }  ][fill={rgb, 255:red, 0; green, 0; blue, 0 }  ][line width=0.75]      (0, 0) circle [x radius= 3.35, y radius= 3.35]   ;
\draw [shift={(157.46,173.07)}, rotate = 269.2] [color={rgb, 255:red, 0; green, 0; blue, 0 }  ][fill={rgb, 255:red, 0; green, 0; blue, 0 }  ][line width=0.75]      (0, 0) circle [x radius= 3.35, y radius= 3.35]   ;
\draw    (157.46,173.07) -- (239.26,173.07) ;
\draw [shift={(239.26,173.07)}, rotate = 0] [color={rgb, 255:red, 0; green, 0; blue, 0 }  ][fill={rgb, 255:red, 0; green, 0; blue, 0 }  ][line width=0.75]      (0, 0) circle [x radius= 3.35, y radius= 3.35]   ;
\draw [shift={(157.46,173.07)}, rotate = 0] [color={rgb, 255:red, 0; green, 0; blue, 0 }  ][fill={rgb, 255:red, 0; green, 0; blue, 0 }  ][line width=0.75]      (0, 0) circle [x radius= 3.35, y radius= 3.35]   ;
\draw    (239.26,173.07) -- (238.59,125.07) ;
\draw [shift={(238.59,125.07)}, rotate = 269.2] [color={rgb, 255:red, 0; green, 0; blue, 0 }  ][fill={rgb, 255:red, 0; green, 0; blue, 0 }  ][line width=0.75]      (0, 0) circle [x radius= 3.35, y radius= 3.35]   ;
\draw [shift={(239.26,173.07)}, rotate = 269.2] [color={rgb, 255:red, 0; green, 0; blue, 0 }  ][fill={rgb, 255:red, 0; green, 0; blue, 0 }  ][line width=0.75]      (0, 0) circle [x radius= 3.35, y radius= 3.35]   ;
\draw    (239.26,173.07) -- (321.06,173.07) ;
\draw [shift={(321.06,173.07)}, rotate = 0] [color={rgb, 255:red, 0; green, 0; blue, 0 }  ][fill={rgb, 255:red, 0; green, 0; blue, 0 }  ][line width=0.75]      (0, 0) circle [x radius= 3.35, y radius= 3.35]   ;
\draw [shift={(239.26,173.07)}, rotate = 0] [color={rgb, 255:red, 0; green, 0; blue, 0 }  ][fill={rgb, 255:red, 0; green, 0; blue, 0 }  ][line width=0.75]      (0, 0) circle [x radius= 3.35, y radius= 3.35]   ;
\draw    (321.06,173.07) -- (320.39,125.07) ;
\draw [shift={(320.39,125.07)}, rotate = 269.2] [color={rgb, 255:red, 0; green, 0; blue, 0 }  ][fill={rgb, 255:red, 0; green, 0; blue, 0 }  ][line width=0.75]      (0, 0) circle [x radius= 3.35, y radius= 3.35]   ;
\draw [shift={(321.06,173.07)}, rotate = 269.2] [color={rgb, 255:red, 0; green, 0; blue, 0 }  ][fill={rgb, 255:red, 0; green, 0; blue, 0 }  ][line width=0.75]      (0, 0) circle [x radius= 3.35, y radius= 3.35]   ;
\draw    (321.06,173.07) -- (402.85,173.07) ;
\draw [shift={(402.85,173.07)}, rotate = 360] [color={rgb, 255:red, 0; green, 0; blue, 0 }  ][fill={rgb, 255:red, 0; green, 0; blue, 0 }  ][line width=0.75]      (0, 0) circle [x radius= 3.35, y radius= 3.35]   ;
\draw [shift={(321.06,173.07)}, rotate = 360] [color={rgb, 255:red, 0; green, 0; blue, 0 }  ][fill={rgb, 255:red, 0; green, 0; blue, 0 }  ][line width=0.75]      (0, 0) circle [x radius= 3.35, y radius= 3.35]   ;
\draw    (402.85,173.07) -- (402.19,125.07) ;
\draw [shift={(402.19,125.07)}, rotate = 269.2] [color={rgb, 255:red, 0; green, 0; blue, 0 }  ][fill={rgb, 255:red, 0; green, 0; blue, 0 }  ][line width=0.75]      (0, 0) circle [x radius= 3.35, y radius= 3.35]   ;
\draw [shift={(402.85,173.07)}, rotate = 269.2] [color={rgb, 255:red, 0; green, 0; blue, 0 }  ][fill={rgb, 255:red, 0; green, 0; blue, 0 }  ][line width=0.75]      (0, 0) circle [x radius= 3.35, y radius= 3.35]   ;
\draw    (402.85,173.07) -- (484.65,173.07) ;
\draw [shift={(484.65,173.07)}, rotate = 0] [color={rgb, 255:red, 0; green, 0; blue, 0 }  ][fill={rgb, 255:red, 0; green, 0; blue, 0 }  ][line width=0.75]      (0, 0) circle [x radius= 3.35, y radius= 3.35]   ;
\draw [shift={(402.85,173.07)}, rotate = 0] [color={rgb, 255:red, 0; green, 0; blue, 0 }  ][fill={rgb, 255:red, 0; green, 0; blue, 0 }  ][line width=0.75]      (0, 0) circle [x radius= 3.35, y radius= 3.35]   ;
\draw    (483.98,125.07) -- (484.65,173.07) ;
\draw [shift={(484.65,173.07)}, rotate = 89.2] [color={rgb, 255:red, 0; green, 0; blue, 0 }  ][fill={rgb, 255:red, 0; green, 0; blue, 0 }  ][line width=0.75]      (0, 0) circle [x radius= 3.35, y radius= 3.35]   ;
\draw [shift={(483.98,125.07)}, rotate = 89.2] [color={rgb, 255:red, 0; green, 0; blue, 0 }  ][fill={rgb, 255:red, 0; green, 0; blue, 0 }  ][line width=0.75]      (0, 0) circle [x radius= 3.35, y radius= 3.35]   ;
\draw    (484.65,173.07) -- (563.45,173.07) ;
\draw [shift={(566.45,173.07)}, rotate = 180] [fill={rgb, 255:red, 0; green, 0; blue, 0 }  ][line width=0.08]  [draw opacity=0] (8.93,-4.29) -- (0,0) -- (8.93,4.29) -- cycle    ;
\draw [shift={(484.65,173.07)}, rotate = 0] [color={rgb, 255:red, 0; green, 0; blue, 0 }  ][fill={rgb, 255:red, 0; green, 0; blue, 0 }  ][line width=0.75]      (0, 0) circle [x radius= 3.35, y radius= 3.35]   ;
\draw    (197.36,44.74) -- (197.69,125.07) ;
\draw [shift={(197.69,125.07)}, rotate = 89.76] [color={rgb, 255:red, 0; green, 0; blue, 0 }  ][fill={rgb, 255:red, 0; green, 0; blue, 0 }  ][line width=0.75]      (0, 0) circle [x radius= 3.35, y radius= 3.35]   ;
\draw [shift={(197.36,44.74)}, rotate = 89.76] [color={rgb, 255:red, 0; green, 0; blue, 0 }  ][fill={rgb, 255:red, 0; green, 0; blue, 0 }  ][line width=0.75]      (0, 0) circle [x radius= 3.35, y radius= 3.35]   ;
\draw    (279.16,44.74) -- (279.49,125.07) ;
\draw [shift={(279.49,125.07)}, rotate = 89.76] [color={rgb, 255:red, 0; green, 0; blue, 0 }  ][fill={rgb, 255:red, 0; green, 0; blue, 0 }  ][line width=0.75]      (0, 0) circle [x radius= 3.35, y radius= 3.35]   ;
\draw [shift={(279.16,44.74)}, rotate = 89.76] [color={rgb, 255:red, 0; green, 0; blue, 0 }  ][fill={rgb, 255:red, 0; green, 0; blue, 0 }  ][line width=0.75]      (0, 0) circle [x radius= 3.35, y radius= 3.35]   ;
\draw    (360.95,44.74) -- (361.29,125.07) ;
\draw [shift={(361.29,125.07)}, rotate = 89.76] [color={rgb, 255:red, 0; green, 0; blue, 0 }  ][fill={rgb, 255:red, 0; green, 0; blue, 0 }  ][line width=0.75]      (0, 0) circle [x radius= 3.35, y radius= 3.35]   ;
\draw [shift={(360.95,44.74)}, rotate = 89.76] [color={rgb, 255:red, 0; green, 0; blue, 0 }  ][fill={rgb, 255:red, 0; green, 0; blue, 0 }  ][line width=0.75]      (0, 0) circle [x radius= 3.35, y radius= 3.35]   ;
\draw    (442.75,44.74) -- (443.08,125.07) ;
\draw [shift={(443.08,125.07)}, rotate = 89.76] [color={rgb, 255:red, 0; green, 0; blue, 0 }  ][fill={rgb, 255:red, 0; green, 0; blue, 0 }  ][line width=0.75]      (0, 0) circle [x radius= 3.35, y radius= 3.35]   ;
\draw [shift={(442.75,44.74)}, rotate = 89.76] [color={rgb, 255:red, 0; green, 0; blue, 0 }  ][fill={rgb, 255:red, 0; green, 0; blue, 0 }  ][line width=0.75]      (0, 0) circle [x radius= 3.35, y radius= 3.35]   ;
\draw    (524.55,44.74) -- (524.88,125.07) ;
\draw [shift={(524.88,125.07)}, rotate = 89.76] [color={rgb, 255:red, 0; green, 0; blue, 0 }  ][fill={rgb, 255:red, 0; green, 0; blue, 0 }  ][line width=0.75]      (0, 0) circle [x radius= 3.35, y radius= 3.35]   ;
\draw [shift={(524.55,44.74)}, rotate = 89.76] [color={rgb, 255:red, 0; green, 0; blue, 0 }  ][fill={rgb, 255:red, 0; green, 0; blue, 0 }  ][line width=0.75]      (0, 0) circle [x radius= 3.35, y radius= 3.35]   ;
\draw   (279.16,44.74) .. controls (279.16,23.24) and (297.47,5.81) .. (320.06,5.81) .. controls (342.64,5.81) and (360.95,23.24) .. (360.95,44.74) .. controls (360.95,66.24) and (342.64,83.67) .. (320.06,83.67) .. controls (297.47,83.67) and (279.16,66.24) .. (279.16,44.74) -- cycle ;
\draw   (442.75,44.74) .. controls (442.75,23.24) and (461.06,5.81) .. (483.65,5.81) .. controls (506.24,5.81) and (524.55,23.24) .. (524.55,44.74) .. controls (524.55,66.24) and (506.24,83.67) .. (483.65,83.67) .. controls (461.06,83.67) and (442.75,66.24) .. (442.75,44.74) -- cycle ;

\draw (84.57,34.52) node [anchor=north west][inner sep=0.75pt]  [font=\small]  {$u_{-1}$};
\draw (203.17,33.72) node [anchor=north west][inner sep=0.75pt]    {$v_{-1}$};
\draw (254.83,33.58) node [anchor=north west][inner sep=0.75pt]    {$u_{0}$};
\draw (368.05,34.38) node [anchor=north west][inner sep=0.75pt]    {$v_{0}$};
\draw (420.15,32.91) node [anchor=north west][inner sep=0.75pt]    {$u_{1}$};
\draw (532.7,33.72) node [anchor=north west][inner sep=0.75pt]    {$v_{1}$};
\draw (144.32,29.34) node [anchor=north west][inner sep=0.75pt]    {$G_{-1}$};
\draw (311.98,29.87) node [anchor=north west][inner sep=0.75pt]    {$G_{0}$};
\draw (476.39,30.4) node [anchor=north west][inner sep=0.75pt]    {$G_{1}$};
\draw (312,101.9) node [anchor=north west][inner sep=0.75pt]    {$x_{0}$};
\draw (393.5,100.9) node [anchor=north west][inner sep=0.75pt]    {$x_{1}$};
\draw (475.5,102.9) node [anchor=north west][inner sep=0.75pt]    {$x_{2}$};
\draw (228,101.4) node [anchor=north west][inner sep=0.75pt]    {$x_{-1}$};
\draw (146.5,101.9) node [anchor=north west][inner sep=0.75pt]    {$x_{-2}$};
\draw (313.67,180.07) node [anchor=north west][inner sep=0.75pt]    {$y_{0}$};
\draw (395.33,179.4) node [anchor=north west][inner sep=0.75pt]    {$y_{1}$};
\draw (476,180.07) node [anchor=north west][inner sep=0.75pt]    {$y_{2}$};
\draw (227.67,180.23) node [anchor=north west][inner sep=0.75pt]    {$y_{-1}$};
\draw (145.33,179.4) node [anchor=north west][inner sep=0.75pt]    {$y_{-2}$};
\draw (351,130.73) node [anchor=north west][inner sep=0.75pt]    {$w_{1}$};
\draw (270,131.57) node [anchor=north west][inner sep=0.75pt]    {$w_{0}$};
\draw (184.67,129.9) node [anchor=north west][inner sep=0.75pt]    {$w_{-1}$};
\draw (102.33,130.73) node [anchor=north west][inner sep=0.75pt]    {$w_{-2}$};
\draw (433,130.57) node [anchor=north west][inner sep=0.75pt]    {$w_{2}$};
\draw (515.83,131.07) node [anchor=north west][inner sep=0.75pt]    {$w_{3}$};

\end{tikzpicture}
        \caption{Double ladder with an amalgamation of infinitely many copies of $G$}
        \label{fig7}
    \end{figure}
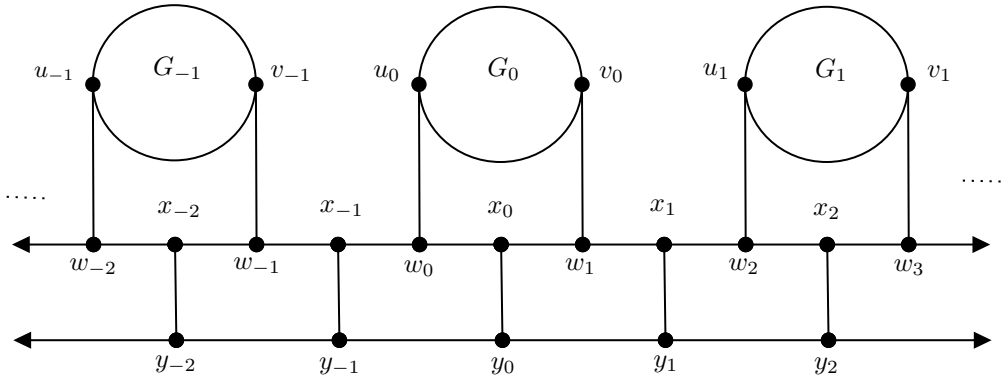

Note that $H$ is cubic and infinite. Furthermore, since each $G_i$ and $L'$ are bridgeless, it follows by construction that $H$ is bridgeless. Now, if $M \subset E(H)$ is a perfect matching of $H$, since $S=\{u_0w_0, v_0w_1\}$ is an edge cut of $H$ and one of the components of $H-S$ is a finite subgraph, the parity of $|M \cap S|$ is equal to the parity of $|V(K)|$, where $K$ is the finite component of $H-S$. Since $H$ is cubic and $S$ has size 2, a parity argument implies that $|M \cap S|$ must be even. Thus, either $S \subset M$ or $S \cap M = \emptyset$. In this way, the perfect matching $M$ induces a perfect matching on $G = G_0 \cup \{u_0v_0\}$. Indeed, if $S \subset M$, let $M' = (M \cap E(G_0)) \cup \{u_0v_0\}$; if $M \cap S = \emptyset$, let $M' = M \cap E(G_0)$. In either case, $M'$ is a perfect matching of $G$.

Now we prove that if $H$ has a $(k,p)$-PM cover, then $G$ also does. Indeed, if $\mathcal{M}=\{M_1, \dots, M_k\}$ is a $(k,p)$-PM cover of $H$, we obtain a collection of perfect matchings $\mathcal{M}'=\{M_1', \dots, M_k'\}$ of $G$. Let $e \in E(G)$. If $e \neq uv$, then $e \in E(H)$. Since $\mathcal{M}$ is a $(k,p)$-PM cover of $H$, the edge $e$ belongs to at most $p$ perfect matchings in the collection $\mathcal{M}$. Since $M_i'$ and $M_i$ agree on all edges of $G_0 = G \setminus \{uv\}$, if $e \neq uv$, then $e$ belongs to at most $p$ matchings in the collection $\mathcal{M}'$. If $e = uv$, since $\mathcal{M}$ is a $(k,p)$-PM cover of $H$, there are at most $p$ matchings in $\mathcal{M}$ containing $S = \{u_0w_0, v_0w_1\}$. Thus, by the definition of $M'_i$, there are at most $p$ matchings in $\mathcal{M}'$ containing $e = uv$. Therefore, $G$ has a $(k,p)$-PM cover. Analogously, it can be proven that if $H$ has a $(k,p)^*$-PM cover, then $G$ also does. 
\end{proof}

\section{Perfect matching theorems}\label{section3}
\paragraph{}

In this section, we prove the main theorems of the paper, using the lemmas from the previous section. We begin with the following general theorem, which we present in terms of PM-covers and will apply to the particular cases of Conjectures~\ref{conjecture1} and~\ref{conjecture2}.

\begin{thm}\label{thm5}
    Let $G$ be an infinite, cubic, bridgeless, connected graph, and let $p,k \in \mathbb{N}_{>0}$ with $p \leq k$. If every finite, cubic, bridgeless, connected graph has a $(k,p)$-PM cover ($(k,p)^*$-PM cover), then $G$ also has a $(k,p)$-PM cover ($(k,p)^*$-PM cover).
\end{thm}
\begin{proof}
    To prove the theorem, it suffices to consider connected graphs. Let $G$ be an infinite, cubic, and bridgeless graph, and let $F$ be a non-principal ultrafilter over $\mathbb{N}$. Consider $\langle H_i\rangle_{i\in\mathbb{N}}$ a sequence of finite, cubic, and bridgeless graphs as in the proof of Lemma \ref{thm1}. Hence, $G$ is the $F$-limit of $\langle H_i\rangle_{i\in\mathbb{N}}$. Since each $H_i$ is finite, cubic, and bridgeless, by the hypothesis of the theorem, each $H_i$ has a collection of $k$ perfect matchings $M_i^1, M_i^2, \dots, M_i^k$ establishing the respective cover. Now consider $M^j$ to be the $F$-limit of the sequence of perfect matchings $\langle M_i^j\rangle_{i\in\mathbb{N}}$ for each $j\in\{ 1, 2, \dots ,k\}$. By Lemma \ref{lemma2}, each $M^j$ is a perfect matching for $G$.
    
First, we prove the $(k,p)$-PM cover case. Let us assume without loss of generality that the collection witnessing this property for each $H_i$ is the one already fixed; that is, $M_i^1, M_i^2, \dots, M_i^k$ are such that each edge $e\in E(H_i)$ belongs to at most $p$ of these perfect matchings. We will show that the collection of perfect matchings $M^1, M^2, \dots, M^k$ obtained via the $F$-limit forms a $(k,p)$-PM cover. Suppose by contradiction that there exist $e\in E(G)$ and distinct indices $j_0, j_1, \dots, j_{p}\in \{ 1, \dots, k\}$ such that $e\in M^{j_0}\cap M^{j_1}\cap \cdots \cap M^{j_{p}}$. Since each $M^{j_l}$ is the $F$-limit of $\langle M_i^{j_l}\rangle_{i\in\mathbb{N}}$, we have $U^{j_l}=\{ i\in\mathbb{N} : e\in M_i^{j_l}\} \in F$ for each $l\in\{ 0, 1, \dots, p \}$. Since the finite intersection of elements of an ultrafilter belongs to the ultrafilter, it follows that:
\[
\left\{ i\in\mathbb{N} : e\in M_i^{j_0}\cap \cdots \cap M_i^{j_p}\right\} = \bigcap_{l=0}^p U^{j_l}\in F.
\]
Thus, there exists an index $i\in\mathbb{N}$ such that $e\in M_i^{j_0}\cap \cdots \cap M_i^{j_p}$. But this is a contradiction, because $H_i$ has a $(k,p)$-PM cover, and therefore $e$ can belong to at most $p$ matchings in the collection $M_i^1, \dots, M_i^k$. Hence, $G$ has a $(k,p)$-PM cover. 

Now, we prove the $(k,p)^*$-PM cover case. Assume without loss of generality that the fixed collection $M_i^1, M_i^2, \dots, M_i^k$ is such that each edge $e\in E(H_i)$ belongs to at least $p$ of these perfect matchings. Fix $e\in E(G)$. Consider $U^j=\{ i\in\mathbb{N} : e\in M_i^j\}$ for $j\in\{ 1, 2, \dots , k\}$. Since there exists $n\in\mathbb{N}$ such that $e\in E(H_i)$ for all $i\geq n$, we have $\{ i\in\mathbb{N} : i\geq n \} \in F.$ Suppose by contradiction that $e$ belongs to at most $p-1$ perfect matchings from the collection $M^1, \dots, M^k$. Let $j_1, \dots, j_{p-1}\in\{ 1, \dots, k\}$ be these indices, so that $e\in M^{j_1}\cap \cdots \cap M^{j_{p-1}}$. Since each $M^{j_l}$ is the $F$-limit of $\langle M_i^{j_l}\rangle_{i\in\mathbb{N}}$, we have $U^{j_l}=\{ i\in\mathbb{N} : e\in M_i^{j_l}\} \in F$ for each $l\in\{ 1, \dots, p-1 \}$. Because the finite intersection of elements in an ultrafilter is also in the ultrafilter, we obtain:
\[
A = \left\{ i\in\mathbb{N} : e\in M_i^{j_1}\cap \cdots \cap M_i^{j_{p-1}} \text{ and } i \geq n\right\} = \{i \in \mathbb{N} : i \geq n\} \cap \bigcap_{l=1}^{p-1} U^{j_l}\in F.
\]
For each $i\in A$, the edge $e$ belongs to at least $p$ perfect matchings from the collection $M_i^1, \dots, M_i^k$ and, in particular, $e\in M_i^{j_1}\cap \cdots \cap M_i^{j_{p-1}}$. Thus, for each $i\in A$, there exists some other matching $M_i^{s_i} \in \{ M_i^1, \dots , M_i^k\}$ containing $e$, where $s_i \notin \{j_1, \dots , j_{p-1}\}$. Let $B = \{ 1, \dots, k\} \setminus \{ j_1, \dots, j_{p-1}\}$ and define $W^j=\{ i\in A : s_i=j\}$ for each $j\in B$. Consequently, $\bigcup_{j\in B} W^j = A \in F.$ Since $B$ is finite and $F$ is an ultrafilter, there exists an index $j\in B$ such that $W^j\in F$. As $e\in M_i^j$ for all $i\in W^j$, we have $W^j \subset \{ i\in\mathbb{N} : e\in M_i^j\}$, which implies that $\{ i\in\mathbb{N} : e\in M_i^j\} \in F$. Therefore, $e\in M^j$ for $j\in B$. It follows that $e\in M^{j_1}\cap \cdots \cap M^{j_{p-1}}\cap M^j$; that is, $e$ belongs to $p$ perfect matchings from the collection $M^1, \dots, M^k$, which is a contradiction. Therefore, $G$ has a $(k,p)^*$-PM cover. This concludes the proof of the theorem.
\end{proof}

As an application of Theorem~\ref{thm5}, we obtain that for Conjectures~\ref{conjecture1} and~\ref{conjecture2}, the finite case is equivalent to the infinite case.

\begin{thm}\label{thm2}
    Conjecture~\ref{conjecture1} holds for finite graphs if and only if it holds for infinite graphs.
\end{thm}
\begin{proof}
    First, we will prove that the finite version implies the infinite version. If Conjecture~\ref{conjecture1} holds for finite graphs, then every finite, cubic, and bridgeless graph has a collection of perfect matchings that is simultaneously a $(6,2)$-PM cover and a $(6,2)^*$-PM cover. Therefore, by Theorem \ref{thm5}, every infinite, cubic, and bridgeless graph $G$ possesses a collection of perfect matchings $\mathcal{M}$ that is both a $(6,2)$-PM cover and a $(6,2)^*$-PM cover.
    
Now, we show that if the conjecture holds for infinite graphs, then it also holds for finite graphs. Consider a finite, connected, cubic, and bridgeless graph $G$. Let $H$ be the infinite, connected, cubic, and bridgeless graph constructed in Lemma \ref{lemma3} from $G$. Since we assume that Conjecture \ref{conjecture1} holds for infinite graphs, there exist 6 perfect matchings $M_1, M_2, \dots, M_6$ of $H$ such that each edge of $H$ belongs to exactly two of them. That is, $H$ has a collection of perfect matchings that is simultaneously a $(6,2)$-PM cover and a $(6,2)^*$-PM cover. By Lemma \ref{lemma3}, each $M_i$ induces a perfect matching on $G$, yielding a collection of matchings that is likewise both a $(6,2)$-PM cover and a $(6,2)^*$-PM cover. Therefore, $G$ has 6 perfect matchings where each edge of $G$ is contained in exactly 2 of them. This concludes the proof of the theorem.
    
\end{proof}
\begin{thm}\label{thm3}
    Conjecture~\ref{conjecture2} holds for finite graphs if and only if it holds for infinite graphs.
    
\end{thm}
\begin{proof}
 First, we will prove that the finite version implies the infinite version. If Conjecture~\ref{conjecture2} holds for finite graphs, then every finite, cubic, and bridgeless graph has three perfect matchings $M_1, M_2, M_3$ such that $M_1 \cap M_2 \cap M_3 = \emptyset$. This means that each edge belongs to at most two of the three perfect matchings; that is, $\{M_1, M_2, M_3\}$ is a $(3,2)$-PM cover. Therefore, by Theorem \ref{thm5}, every infinite, cubic, and bridgeless graph $G$ possesses a collection of perfect matchings $\mathcal{M}$ that is a $(3,2)$-PM cover.

Now, we show that if Conjecture \ref{conjecture2} holds for infinite graphs, then it also holds for finite graphs. Consider a finite, connected, cubic, and bridgeless graph $G$. Let $H$ be the infinite, connected, cubic, and bridgeless graph constructed in Lemma \ref{lemma3} from $G$. Since we assume that Conjecture \ref{conjecture2} holds for infinite graphs, there exist three perfect matchings $M_1, M_2, M_3$ of $H$ such that $M_1 \cap M_2 \cap M_3 = \emptyset$; that is, $H$ has a $(3,2)$-PM cover. By Lemma \ref{lemma3}, the existence of this cover in $H$ implies that $G$ also has a $(3,2)$-PM cover. Therefore, $G$ has three perfect matchings $M_1', M_2', M_3'$ such that $M_1' \cap M_2' \cap M_3' = \emptyset$. This concludes the proof of the theorem.
\end{proof}

Conjecture \ref{conjecture3} does not share the same structural nature as the previous two. It requires finding two perfect matchings whose intersection does not contain an odd edge-cut. To overcome this obstacle, Lemma \ref{thm1} was proved in a more complex manner than strictly necessary, ensuring that the finite graphs in the approximation maintain a more faithful relationship with the infinite graph regarding their cuts. Furthermore, guaranteeing that the intersection of the matchings contains no odd cut leads to a significantly more intricate proof than those we have presented so far.

\begin{thm}\label{thm4}
    If Conjecture~\ref{conjecture3} holds for finite graphs if and only if it holds for infinite graphs.

\end{thm}
\begin{proof}
    Note that it suffices to prove the theorem for connected graphs. Let $G$ be an infinite, cubic, connected, and bridgeless graph, let $F$ be a non-principal ultrafilter over $\mathbb{N}$, and let $\langle H_i\rangle_{i\in\mathbb{N}}$ be the sequence of finite, cubic, and bridgeless graphs constructed in Lemma \ref{thm1} such that $G$ is the $F$-limit of the sequence. Since we assume that Conjecture \ref{conjecture3} holds for finite graphs, for each $H_i$ there exist two perfect matchings $M_i^1$ and $M_i^2$ such that $M_i^1 \cap M_i^2$ contains no odd cuts. By Lemma \ref{lemma2}, the $F$-limit of $\langle M_i^j\rangle_{i\in\mathbb{N}}$ is a perfect matching of $G$, which we denote by $M^j$ for $j \in \{1,2\}$. It remains to prove that $M^1 \cap M^2$ contains no odd cuts.

Suppose by contradiction that there exists an odd cut $L \subset E(G)$ of $G$ such that $L \subset M^1 \cap M^2$. Since $L$ is finite and $L \subset M^1 \cap M^2$, there exists $n \in \mathbb{N}$ such that for all $i \geq n$, we have $L \subset E(G_i)$ and $V(L) \subset V(G_i)$, where $V(L)$ is the set of vertices of $G$ incident to the edges of $L$. That is, since $G = \bigcup_{i\in\mathbb{N}} G_i$ and $G_i \subset G_{i+1}$, there exists $n \in \mathbb{N}$ such that for all $i \geq n$, all edges of $L$ are in $G_i$ and all vertices in $V(L)$ have degree 3 in $G_i$. Note that, since $L$ is a cut in $G$, $L$ is also a cut in $G_i$ for $i \geq n$. Indeed, since $G_i \subset G$ and $L$ is a cut of $G$, there exist $A, B \subset V(G)$ such that $A \cap B = \emptyset$, $A \cup B = V(G)$, and $L = E_G(A,B)$. Thus, it suffices to set $A_i = A \cap V(G_i)$ and $B_i = B \cap V(G_i)$, so that $L$ is a cut of $G_i$ with $L = E_{G_{i}}(A_i, B_i)$. Recall that $G_{m_i}$ is a graph from the ear decomposition such that $G_i \subset G_{m_i}$ and every vertex of $G_i$ has degree 3 in $G_{m_i}$. Hence, we also obtain that $L$ is a cut in $G_{m_i}$.

We now show that $L \subset E(H_i')$ for all $i \geq n$. Recall that the graph $H_i'$ is obtained from $G_{m_i}$ by edge contraction of the paths in $G_{m_i}$ whose internal vertices have degree 2 in $G_{m_i}$ and which are maximal with this property (see Figure \ref{fig2}). Since all vertices in $V(L)$ have degree 3 in $G_i$ for $i \geq n$, they also have degree 3 in $G_{m_i}$ because $G_i \subset G_{m_i}$. Thus, the vertices of $V(L)$ are not internal vertices of any paths to be contracted into edges in the graph $H_i'$. Hence, $L \subset E(H_i')$. Now, we prove that $L \subset E(H_i)$. Observe that since all vertices of $V(L)$ have degree 3 in $G_i$, all neighbors of $V(L)$ also have degree 3 in $G_{m_i}$; that is, if $v \in N_G(V(L))$, then for $i \geq n$ we have $v \in V(G_{m_i})$ and $d_{G_{m_i}}(v) = 3$ (see Figure \ref{fig8}). Therefore, the neighboring vertices of $V(L)$ in $G_{m_i}$ are not internal vertices of the paths to be contracted into edges in $H_i'$. Thus, the edges in the set $E_{G_{m_i}}(V(L), N_G(V(L)))$ will not be contracted in $H_i'$. Consequently, given an edge $e = uv \in L$, there is no path in $G_{m_i}$ connecting $u$ and $v$ such that all its internal vertices have degree 2 in $G_{m_i}$, since such a path would necessarily pass through some neighbor of $u$ and some neighbor of $v$, which have degree 3 in $G_{m_i}$. Thus, if $e \in L \subset E(H_i')$, then $e$ is not a parallel edge in $H_i'$. Therefore, $e \in E(H_i)$, implying that $L \subset E(H_i)$ for all $i \geq n$. 

     \begin{figure}[ht]
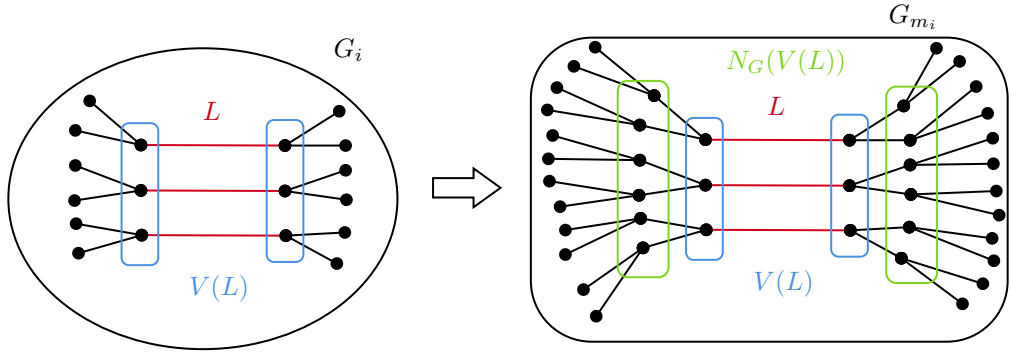

         \centering

\tikzset{every picture/.style={line width=0.75pt}} 


         \caption{Vertices of $L$ and their neighborhoods in $G_{m_i}$}
         \label{fig8}
     \end{figure}

     Now we will prove that $L$ is a cut in $H_i$ for all $i \geq n$. Since $L$ is a cut of $G_{i}$ for all $i \geq n$, there exist $A_{i}, B_{i} \subset V(G_{i})$ such that $A_{i} \cap B_{i} = \emptyset$, $A_{i} \cup B_{i} = V(G_{i})$, and $L = E_{G_{i}}(A_{i}, B_{i})$. Furthermore, note that $A_i \subset A_{m_i}$ and $B_{i} \subset B_{m_i}$. Consider the following sets:
\[
A_{i}' = \{ w \in V(H_i \setminus G_{i}) : \text{there exists a path in } H_i \setminus L \text{ between } w \text{ and a vertex of } A_i \};
\]
\[
B_{i}' = \{ w \in V(H_i \setminus G_{i}) : \text{there exists a path in } H_i \setminus L \text{ between } w \text{ and a vertex of } B_i \};
\]
recalling that $G_i \subset H_i$. We will prove that the sets $\overline{A}_i = A_i \cup A_i'$ and $\overline{B}_i = B_i \cup B_i'$ form a partition of $V(H_i)$, that is, $\overline{A}_i \cap \overline{B}_i = \emptyset$ and $\overline{A}_i \cup \overline{B}_i = V(H_i)$. Moreover, we will show that $E_{H_i}(\overline{A}_i, \overline{B}_i) = L$.

     \begin{claim}\label{claim1}
         $\overline{A}_i\cup \overline{B}_i=V(H_i)$.
     \end{claim}
     \begin{proof}
         Since $H_i$ is connected, each connected component of $H_i \setminus L$ contains a vertex of either $A_i$ or $B_i$, because each component of $H_i \setminus L$ contains at least one vertex from $V(L) \subset A_i \cup B_i$. Thus, for every $v \in V(H_i \setminus G_i)$, there exists a path in $H_i \setminus L$ between $v$ and a vertex of $A_i$ or a vertex of $B_i$. Hence, $v \in A_i' \cup B_i'$. Therefore, $\overline{A}_i \cup \overline{B}_i = V(H_i)$.
        
     \end{proof}
\begin{claim}\label{claim2}
    There is no path in $H_i-L$ between a vertex of $A_i$ and a vertex of $B_i$.
   
\end{claim}
\begin{proof}
    Note initially that if $e \in E(H_i)$, then it can only be of 3 types:
\begin{enumerate}
    \item $e \in E(G_{m_i}) \cap E(H_i)$;
    \item $e \in E(H_i) \setminus E(G_{m_i})$ and $e \in E(H_i') \cap E(H_i)$, that is, $e$ is an edge resulting from the contraction of a path $P_e$ of $G_{m_i}$ and, furthermore, $e$ does not become a parallel edge in $H_i'$;
    \item $e \in E(H_i) \setminus E(G_{m_i})$ and $e \in E(H_i') \setminus E(H_i)$. In this case, there exist $u,v \in V(G_{m_i})$ such that $e$ has one of the following configurations: $e=uw_{uv}'$, $e=uw_{uv}$, $e=vw_{uv}'$, $e=vw_{uv}$, or $e=w_{uv}w'_{uv}$ (see Figure \ref{fig4}).
\end{enumerate}
Suppose by contradiction that there exists a path $S$ in $H_i \setminus L$ connecting a vertex $x \in A_i$ and a vertex $y \in B_i$. We will construct a path $S'$ in $G_{m_i}$ induced by $S$. To do so, we will show that each edge $e \in S$ can be replaced by a path $S_e \subset G_{m_i}$ such that the union of the paths $S_e$ forms the path $S'$. We analyze the three types of edges that can appear in $S$:
\begin{itemize}
    \item If $e \in E(S)$ is of type 1, then let $S_e = \{e\} \subset G_{m_i}$.
    \item If $e \in E(S)$ is of type 2, then there exists a path $Q_e \subset G_{m_i}$ that was contracted into the edge $e$, such that it is maximal with respect to the property that all its internal vertices have degree 2 in $G_{m_i}$. Note that $e=uv \in E(H_i)$ if and only if $Q_e$ connects the vertices $u$ and $v$ in $G_{m_i}$. Thus, let $S_e = Q_e$.
    \item If $e \in E(S)$ is of type 3, then there exist $u,v \in V(G_{m_i})$ such that:
    \[
    e \in N = \{uw_{uv}', uw_{uv}, vw_{uv}', vw_{uv}, w_{uv}w'_{uv}\},
    \]
    see Figure \ref{fig4}. By construction, $u,v \in V(H_i')$. Furthermore, $w_{uv}, w_{uv}' \notin A_i \cup B_i$, since $w_{uv}, w_{uv}' \notin V(G_{m_i})$ and $A_i \cup B_i \subset V(G_{m_i})$. Since $e \in E(S)$, $e \in N$, and $S$ is a path connecting a vertex of $A_i$ to a vertex of $B_i$, it follows that $S$ must contain the vertices $u$ and $v$. Additionally, consider the induced subgraph $K = H_i[\{u, v, w_{uv}, w'_{uv}\}]$ (like the rightmost graph in Figure \ref{fig9} $(a)$). Then, $S \cap K$ is a path containing $e$, see Figure~\ref{fig9} $(b)$.
    
        \begin{figure}[ht]
            \centering

\tikzset{every picture/.style={line width=0.75pt}} 

\begin{tikzpicture}[x=0.6pt,y=0.6pt,yscale=-1,xscale=1]

\draw    (40.97,64.57) -- (82.72,65.07) ;
\draw [shift={(82.72,65.07)}, rotate = 0.69] [color={rgb, 255:red, 0; green, 0; blue, 0 }  ][fill={rgb, 255:red, 0; green, 0; blue, 0 }  ][line width=0.75]      (0, 0) circle [x radius= 3.35, y radius= 3.35]   ;
\draw [shift={(40.97,64.57)}, rotate = 0.69] [color={rgb, 255:red, 0; green, 0; blue, 0 }  ][fill={rgb, 255:red, 0; green, 0; blue, 0 }  ][line width=0.75]      (0, 0) circle [x radius= 3.35, y radius= 3.35]   ;
\draw    (82.72,65.07) -- (124.47,65.57) ;
\draw [shift={(124.47,65.57)}, rotate = 0.69] [color={rgb, 255:red, 0; green, 0; blue, 0 }  ][fill={rgb, 255:red, 0; green, 0; blue, 0 }  ][line width=0.75]      (0, 0) circle [x radius= 3.35, y radius= 3.35]   ;
\draw [shift={(82.72,65.07)}, rotate = 0.69] [color={rgb, 255:red, 0; green, 0; blue, 0 }  ][fill={rgb, 255:red, 0; green, 0; blue, 0 }  ][line width=0.75]      (0, 0) circle [x radius= 3.35, y radius= 3.35]   ;
\draw    (82.72,65.07) -- (83.63,36.23) ;
\draw [shift={(83.63,36.23)}, rotate = 271.82] [color={rgb, 255:red, 0; green, 0; blue, 0 }  ][fill={rgb, 255:red, 0; green, 0; blue, 0 }  ][line width=0.75]      (0, 0) circle [x radius= 3.35, y radius= 3.35]   ;
\draw [color={rgb, 255:red, 0; green, 0; blue, 0 }  ,draw opacity=1 ]   (40.97,64.57) .. controls (50.23,39.87) and (71.43,35.87) .. (83.63,36.23) ;
\draw [shift={(40.97,64.57)}, rotate = 290.56] [color={rgb, 255:red, 0; green, 0; blue, 0 }  ,draw opacity=1 ][fill={rgb, 255:red, 0; green, 0; blue, 0 }  ,fill opacity=1 ][line width=0.75]      (0, 0) circle [x radius= 3.35, y radius= 3.35]   ;
\draw [color={rgb, 255:red, 0; green, 0; blue, 0 }  ,draw opacity=1 ]   (83.63,36.23) .. controls (107.43,36.27) and (123.03,53.47) .. (124.47,65.57) ;
\draw [shift={(83.63,36.23)}, rotate = 0.08] [color={rgb, 255:red, 0; green, 0; blue, 0 }  ,draw opacity=1 ][fill={rgb, 255:red, 0; green, 0; blue, 0 }  ,fill opacity=1 ][line width=0.75]      (0, 0) circle [x radius= 3.35, y radius= 3.35]   ;
\draw    (204.47,65.17) -- (246.22,65.67) ;
\draw [shift={(246.22,65.67)}, rotate = 0.69] [color={rgb, 255:red, 0; green, 0; blue, 0 }  ][fill={rgb, 255:red, 0; green, 0; blue, 0 }  ][line width=0.75]      (0, 0) circle [x radius= 3.35, y radius= 3.35]   ;
\draw [shift={(204.47,65.17)}, rotate = 0.69] [color={rgb, 255:red, 0; green, 0; blue, 0 }  ][fill={rgb, 255:red, 0; green, 0; blue, 0 }  ][line width=0.75]      (0, 0) circle [x radius= 3.35, y radius= 3.35]   ;
\draw [color={rgb, 255:red, 208; green, 2; blue, 27 }  ,draw opacity=1 ]   (246.22,65.67) -- (287.97,66.17) ;
\draw [shift={(287.97,66.17)}, rotate = 0.69] [color={rgb, 255:red, 208; green, 2; blue, 27 }  ,draw opacity=1 ][fill={rgb, 255:red, 208; green, 2; blue, 27 }  ,fill opacity=1 ][line width=0.75]      (0, 0) circle [x radius= 3.35, y radius= 3.35]   ;
\draw [shift={(246.22,65.67)}, rotate = 0.69] [color={rgb, 255:red, 208; green, 2; blue, 27 }  ,draw opacity=1 ][fill={rgb, 255:red, 208; green, 2; blue, 27 }  ,fill opacity=1 ][line width=0.75]      (0, 0) circle [x radius= 3.35, y radius= 3.35]   ;
\draw [color={rgb, 255:red, 208; green, 2; blue, 27 }  ,draw opacity=1 ]   (246.22,65.67) -- (247.13,36.83) ;
\draw [shift={(247.13,36.83)}, rotate = 271.82] [color={rgb, 255:red, 208; green, 2; blue, 27 }  ,draw opacity=1 ][fill={rgb, 255:red, 208; green, 2; blue, 27 }  ,fill opacity=1 ][line width=0.75]      (0, 0) circle [x radius= 3.35, y radius= 3.35]   ;
\draw [color={rgb, 255:red, 208; green, 2; blue, 27 }  ,draw opacity=1 ]   (204.47,65.17) .. controls (213.73,40.47) and (234.93,36.47) .. (247.13,36.83) ;
\draw [shift={(204.47,65.17)}, rotate = 290.56] [color={rgb, 255:red, 208; green, 2; blue, 27 }  ,draw opacity=1 ][fill={rgb, 255:red, 208; green, 2; blue, 27 }  ,fill opacity=1 ][line width=0.75]      (0, 0) circle [x radius= 3.35, y radius= 3.35]   ;
\draw    (247.13,36.83) .. controls (270.93,36.87) and (286.53,54.07) .. (287.97,66.17) ;
\draw   (149.47,62.78) -- (164.77,62.78) -- (164.77,58.9) -- (174.97,66.65) -- (164.77,74.4) -- (164.77,70.53) -- (149.47,70.53) -- cycle ;
\draw    (532.1,66) -- (559.89,66.17) ;
\draw [shift={(559.89,66.17)}, rotate = 0.34] [color={rgb, 255:red, 0; green, 0; blue, 0 }  ][fill={rgb, 255:red, 0; green, 0; blue, 0 }  ][line width=0.75]      (0, 0) circle [x radius= 3.35, y radius= 3.35]   ;
\draw [shift={(532.1,66)}, rotate = 0.34] [color={rgb, 255:red, 0; green, 0; blue, 0 }  ][fill={rgb, 255:red, 0; green, 0; blue, 0 }  ][line width=0.75]      (0, 0) circle [x radius= 3.35, y radius= 3.35]   ;
\draw [color={rgb, 255:red, 0; green, 0; blue, 0 }  ,draw opacity=1 ]   (559.89,66.17) -- (584.99,66.24) ;
\draw [shift={(584.99,66.24)}, rotate = 0.17] [color={rgb, 255:red, 0; green, 0; blue, 0 }  ,draw opacity=1 ][fill={rgb, 255:red, 0; green, 0; blue, 0 }  ,fill opacity=1 ][line width=0.75]      (0, 0) circle [x radius= 3.35, y radius= 3.35]   ;
\draw [shift={(559.89,66.17)}, rotate = 0.17] [color={rgb, 255:red, 0; green, 0; blue, 0 }  ,draw opacity=1 ][fill={rgb, 255:red, 0; green, 0; blue, 0 }  ,fill opacity=1 ][line width=0.75]      (0, 0) circle [x radius= 3.35, y radius= 3.35]   ;
\draw [color={rgb, 255:red, 0; green, 0; blue, 0 }  ,draw opacity=1 ]   (584.99,66.24) -- (610.1,66.31) ;
\draw [shift={(610.1,66.31)}, rotate = 0.17] [color={rgb, 255:red, 0; green, 0; blue, 0 }  ,draw opacity=1 ][fill={rgb, 255:red, 0; green, 0; blue, 0 }  ,fill opacity=1 ][line width=0.75]      (0, 0) circle [x radius= 3.35, y radius= 3.35]   ;
\draw [shift={(584.99,66.24)}, rotate = 0.17] [color={rgb, 255:red, 0; green, 0; blue, 0 }  ,draw opacity=1 ][fill={rgb, 255:red, 0; green, 0; blue, 0 }  ,fill opacity=1 ][line width=0.75]      (0, 0) circle [x radius= 3.35, y radius= 3.35]   ;
\draw [color={rgb, 255:red, 0; green, 0; blue, 0 }  ,draw opacity=1 ]   (610.1,66.31) -- (635.2,66.38) ;
\draw [shift={(635.2,66.38)}, rotate = 0.17] [color={rgb, 255:red, 0; green, 0; blue, 0 }  ,draw opacity=1 ][fill={rgb, 255:red, 0; green, 0; blue, 0 }  ,fill opacity=1 ][line width=0.75]      (0, 0) circle [x radius= 3.35, y radius= 3.35]   ;
\draw [shift={(610.1,66.31)}, rotate = 0.17] [color={rgb, 255:red, 0; green, 0; blue, 0 }  ,draw opacity=1 ][fill={rgb, 255:red, 0; green, 0; blue, 0 }  ,fill opacity=1 ][line width=0.75]      (0, 0) circle [x radius= 3.35, y radius= 3.35]   ;
\draw   (314.67,62.08) -- (329.97,62.08) -- (329.97,58.2) -- (340.17,65.95) -- (329.97,73.7) -- (329.97,69.83) -- (314.67,69.83) -- cycle ;
\draw [color={rgb, 255:red, 208; green, 2; blue, 27 }  ,draw opacity=1 ]   (547.27,45) -- (532.1,66) ;
\draw [shift={(532.1,66)}, rotate = 125.84] [color={rgb, 255:red, 208; green, 2; blue, 27 }  ,draw opacity=1 ][fill={rgb, 255:red, 208; green, 2; blue, 27 }  ,fill opacity=1 ][line width=0.75]      (0, 0) circle [x radius= 3.35, y radius= 3.35]   ;
\draw [shift={(547.27,45)}, rotate = 125.84] [color={rgb, 255:red, 208; green, 2; blue, 27 }  ,draw opacity=1 ][fill={rgb, 255:red, 208; green, 2; blue, 27 }  ,fill opacity=1 ][line width=0.75]      (0, 0) circle [x radius= 3.35, y radius= 3.35]   ;
\draw [color={rgb, 255:red, 208; green, 2; blue, 27 }  ,draw opacity=1 ]   (575.93,30.33) -- (547.27,45) ;
\draw [shift={(547.27,45)}, rotate = 152.9] [color={rgb, 255:red, 208; green, 2; blue, 27 }  ,draw opacity=1 ][fill={rgb, 255:red, 208; green, 2; blue, 27 }  ,fill opacity=1 ][line width=0.75]      (0, 0) circle [x radius= 3.35, y radius= 3.35]   ;
\draw [shift={(575.93,30.33)}, rotate = 152.9] [color={rgb, 255:red, 208; green, 2; blue, 27 }  ,draw opacity=1 ][fill={rgb, 255:red, 208; green, 2; blue, 27 }  ,fill opacity=1 ][line width=0.75]      (0, 0) circle [x radius= 3.35, y radius= 3.35]   ;
\draw [color={rgb, 255:red, 208; green, 2; blue, 27 }  ,draw opacity=1 ]   (600.6,31.67) -- (575.93,30.33) ;
\draw [shift={(575.93,30.33)}, rotate = 183.09] [color={rgb, 255:red, 208; green, 2; blue, 27 }  ,draw opacity=1 ][fill={rgb, 255:red, 208; green, 2; blue, 27 }  ,fill opacity=1 ][line width=0.75]      (0, 0) circle [x radius= 3.35, y radius= 3.35]   ;
\draw [shift={(600.6,31.67)}, rotate = 183.09] [color={rgb, 255:red, 208; green, 2; blue, 27 }  ,draw opacity=1 ][fill={rgb, 255:red, 208; green, 2; blue, 27 }  ,fill opacity=1 ][line width=0.75]      (0, 0) circle [x radius= 3.35, y radius= 3.35]   ;
\draw [color={rgb, 255:red, 208; green, 2; blue, 27 }  ,draw opacity=1 ]   (600.6,31.67) -- (617.93,43.67) ;
\draw [shift={(617.93,43.67)}, rotate = 34.7] [color={rgb, 255:red, 208; green, 2; blue, 27 }  ,draw opacity=1 ][fill={rgb, 255:red, 208; green, 2; blue, 27 }  ,fill opacity=1 ][line width=0.75]      (0, 0) circle [x radius= 3.35, y radius= 3.35]   ;
\draw [shift={(600.6,31.67)}, rotate = 34.7] [color={rgb, 255:red, 208; green, 2; blue, 27 }  ,draw opacity=1 ][fill={rgb, 255:red, 208; green, 2; blue, 27 }  ,fill opacity=1 ][line width=0.75]      (0, 0) circle [x radius= 3.35, y radius= 3.35]   ;
\draw [color={rgb, 255:red, 208; green, 2; blue, 27 }  ,draw opacity=1 ]   (617.93,43.67) -- (635.2,66.38) ;
\draw [shift={(635.2,66.38)}, rotate = 52.76] [color={rgb, 255:red, 208; green, 2; blue, 27 }  ,draw opacity=1 ][fill={rgb, 255:red, 208; green, 2; blue, 27 }  ,fill opacity=1 ][line width=0.75]      (0, 0) circle [x radius= 3.35, y radius= 3.35]   ;
\draw [shift={(617.93,43.67)}, rotate = 52.76] [color={rgb, 255:red, 208; green, 2; blue, 27 }  ,draw opacity=1 ][fill={rgb, 255:red, 208; green, 2; blue, 27 }  ,fill opacity=1 ][line width=0.75]      (0, 0) circle [x radius= 3.35, y radius= 3.35]   ;
\draw    (375,65.6) -- (416.75,66.1) ;
\draw [shift={(375,65.6)}, rotate = 0.69] [color={rgb, 255:red, 0; green, 0; blue, 0 }  ][fill={rgb, 255:red, 0; green, 0; blue, 0 }  ][line width=0.75]      (0, 0) circle [x radius= 3.35, y radius= 3.35]   ;
\draw    (416.75,66.1) -- (458.5,66.6) ;
\draw [shift={(458.5,66.6)}, rotate = 0.69] [color={rgb, 255:red, 0; green, 0; blue, 0 }  ][fill={rgb, 255:red, 0; green, 0; blue, 0 }  ][line width=0.75]      (0, 0) circle [x radius= 3.35, y radius= 3.35]   ;
\draw [color={rgb, 255:red, 0; green, 0; blue, 0 }  ,draw opacity=1 ]   (375,65.6) .. controls (384.27,40.9) and (405.47,36.9) .. (417.67,37.27) ;
\draw [color={rgb, 255:red, 0; green, 0; blue, 0 }  ,draw opacity=1 ]   (417.67,37.27) .. controls (441.47,37.3) and (457.07,54.5) .. (458.5,66.6) ;
\draw   (482.27,63.14) -- (497.57,63.14) -- (497.57,59.27) -- (507.77,67.02) -- (497.57,74.77) -- (497.57,70.89) -- (482.27,70.89) -- cycle ;

\draw (19.97,53.47) node [anchor=north west][inner sep=0.75pt]    {$u$};
\draw (131.97,55.47) node [anchor=north west][inner sep=0.75pt]    {$v$};
\draw (70.63,74.63) node [anchor=north west][inner sep=0.75pt]    {$w_{uv}$};
\draw (69.3,8.63) node [anchor=north west][inner sep=0.75pt]    {$w'_{uv}$};
\draw (59.93,98.67) node [anchor=north west][inner sep=0.75pt]    {$K\subset H_{i}$};
\draw (183.47,54.07) node [anchor=north west][inner sep=0.75pt]    {$u$};
\draw (294.97,55.57) node [anchor=north west][inner sep=0.75pt]    {$v$};
\draw (234.13,75.23) node [anchor=north west][inner sep=0.75pt]    {$w_{uv}$};
\draw (232.8,9.23) node [anchor=north west][inner sep=0.75pt]    {$w'_{uv}$};
\draw (210.07,99.6) node [anchor=north west][inner sep=0.75pt]  [color={rgb, 255:red, 208; green, 2; blue, 27 }  ,opacity=1 ]  {$S\cap K\subset H_{i}$};
\draw (511.8,54.57) node [anchor=north west][inner sep=0.75pt]    {$u$};
\draw (645.3,57.07) node [anchor=north west][inner sep=0.75pt]    {$v$};
\draw (410.7,97.67) node [anchor=north west][inner sep=0.75pt]    {$H_{i} '$};
\draw (354,54.5) node [anchor=north west][inner sep=0.75pt]    {$u$};
\draw (466,56.5) node [anchor=north west][inner sep=0.75pt]    {$v$};
\draw (408,12) node [anchor=north west][inner sep=0.75pt]    {$e'_{uv}$};
\draw (409.33,68.33) node [anchor=north west][inner sep=0.75pt]    {$e_{u}{}_{v}$};
\draw (572.4,93.6) node [anchor=north west][inner sep=0.75pt]    {$G_{m}{}_{_{i}}$};
\draw (580.27,3.07) node [anchor=north west][inner sep=0.75pt]  [color={rgb, 255:red, 208; green, 2; blue, 27 }  ,opacity=1 ]  {$S_{e}$};
\draw (71.33,126.73) node [anchor=north west][inner sep=0.75pt]    {$( a)$};
\draw (236,126.73) node [anchor=north west][inner sep=0.75pt]    {$( b)$};
\draw (412,127.73) node [anchor=north west][inner sep=0.75pt]    {$( c)$};
\draw (574,127.73) node [anchor=north west][inner sep=0.75pt]    {$( d)$};

\end{tikzpicture}
            \caption{Construction of $S_e$ in the case of former parallel edges}
            \label{fig9}
        \end{figure}
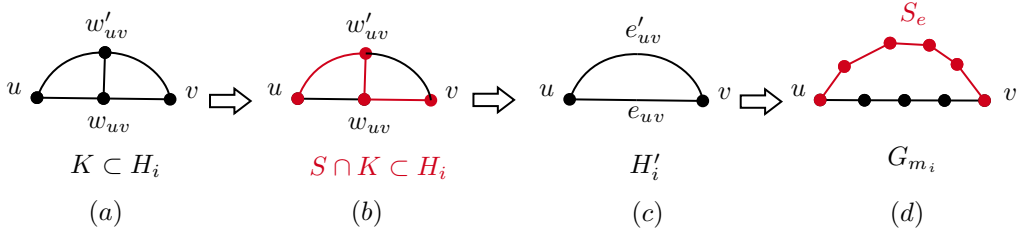

Note that $L \cap E(K) = \emptyset$, since $L \subset E(G_{m_i})$ and $E(K) \cap E(G_{m_i}) = \emptyset$. Recall that $w_{uv}'$ and $w_{uv}$ were vertices added to the edges $e_{uv}'$ and $e_{uv}$ of $H_i'$, respectively, to eliminate the parallel edges of $H_i'$ (see Figure \ref{fig9} $(c)$). Then, since $L \cap E(K) = \emptyset$, we have $e_{uv}', e_{uv} \notin L$. Recall also that the edge $e_{uv}$ is the result of the contraction of a path $Q_{e_{uv}}$ of $G_{m_i}$. Since $L \subset E(H_i')$, none of its edges were contracted or subdivided during the constructions of $H_i'$ and $H_i$. Hence, $L \cap E(Q_{e_{uv}}) = \emptyset$. Thus, let $S_e = Q_{e_{uv}}$ (see Figure~\ref{fig9}~$(d)$). Note that this choice is independent of the configuration of $e \in N$. Note also that we are replacing the path $S \cap K$ (which connects $u$ and $v$ in $H_i \setminus L$) with $Q_{e_{uv}}$ (which connects $u$ and $v$ in $G_{m_i} \setminus L$).
 \end{itemize}

Now, to construct the trail $S'' \subset G_{m_i}$ from $S \subset H_i$, it suffices to set $S' = \bigcup_{e\in E(S)} S_e$. By construction, $S''$ is a trail in $G_{m_i}$ connecting $x \in A_i$ and $y \in B_i$. Consider $S'\subset S''$ a path connecting $x \in A_i$ and $y \in B_i$. We will show that $E(S') \cap L = \emptyset$. As constructed, each $S_e$ satisfies $E(S_e) \cap L = \emptyset$, which implies $E(S') \cap L = \emptyset$. But this is a contradiction, because $L$ is a cut in $G_{m_i}$ with shores $A_{m_i}$ and $B_{m_i}$, where $A_i \subset A_{m_i}$ and $B_{i} \subset B_{m_i}$, which prevents the existence of a path in $G_{m_i} \setminus L$ joining $A_i$ and $B_i$. Therefore, there can be no path in $H_i \setminus L$ connecting a vertex of $A_i$ to a vertex of $B_i$. This completes the proof of the claim.

\end{proof}
     \begin{claim}\label{claim3}
         $\overline{A}_i\cap \overline{B}_i=\emptyset$.
     \end{claim}
     \begin{proof}
         Suppose by contradiction that $\overline{A}_i \cap \overline{B}_i \neq \emptyset$, that is, there exists a vertex $z \in \overline{A}_i \cap \overline{B}_i$. Since $A_i$ and $B_i$ form a partition of $V(G_i)$, by the definitions of $A_i'$ and $B_i'$, it follows that $A_i \cap (B_i \cup B_i') = \emptyset$ and $B_i \cap (A_i \cup A_i') = \emptyset$. Thus, we must have $z \in A_i' \cap B_i'$. By the definition of these sets, there exists a path $Q$ in $H_i \setminus L$ connecting $z$ to some vertex $x \in A_i$, and a path $Q'$ in $H_i \setminus L$ connecting $z$ to some vertex $y \in B_i$. By concatenating these paths, we obtain a path $S$ in $H_i \setminus L$ connecting $x \in A_i$ and $y \in B_i$, which contradicts Claim \ref{claim2}.
     \end{proof}
     \begin{claim}\label{claim4}
         $E_{H_i}(\overline{A}_i, \overline{B}_i)=L$.
     \end{claim}
     \begin{proof}
         Since $L$ was taken as a cut in $G_i$, every edge of $L$ has one vertex in $A_i \subset \overline{A}_i$ and the other vertex in $B_i \subset \overline{B}_i$. Thus, $L \subset E_{H_i}(\overline{A}_i, \overline{B}_i)$.

Conversely, let $e \in E_{H_i}(\overline{A}_i, \overline{B}_i)$ and suppose by contradiction that $e \notin L$. Let $e = uv$ with $u \in \overline{A}_i$ and $v \in \overline{B}_i$. By the definition of $\overline{A}_i$, since $u \in \overline{A}_i$, there exist a vertex $x \in A_i$ and a path $Q \subset H_i \setminus L$ connecting $u$ and $x$. Similarly, by the definition of $\overline{B}_i$, since $v \in \overline{B}_i$, there exist a vertex $y \in B_i$ and a path $Q' \subset H_i \setminus L$ connecting $v$ and $y$. Since the edge $e = uv$ does not belong to $L$, we can concatenate the path $Q$, the edge $e$, and the path $Q'$ to obtain a path in $H_i \setminus L$ connecting $x \in A_i$ and $y \in B_i$. However, this contradicts Claim~\ref{claim2}. Hence, $E_{H_i}(\overline{A}_i, \overline{B}_i) \subset L$, which implies that $E_{H_i}(\overline{A}_i, \overline{B}_i) = L$.
     \end{proof}

     Since $L$ was taken as a cut in $G_i$, every edge of $L$ has one vertex in $A_i \subset \overline{A}_i$ and the other vertex in $B_i \subset \overline{B}_i$. Thus, $L \subset E_{H_i}(\overline{A}_i, \overline{B}_i)$.

Conversely, let $e \in E_{H_i}(\overline{A}_i, \overline{B}_i)$ and suppose by contradiction that $e \notin L$. Let $e = uv$ with $u \in \overline{A}_i$ and $v \in \overline{B}_i$. By the definition of $\overline{A}_i$, since $u \in \overline{A}_i$, there exist a vertex $x \in A_i$ and a path $Q \subset H_i \setminus L$ connecting $u$ and $x$. Similarly, by the definition of $\overline{B}_i$, since $v \in \overline{B}_i$, there exist a vertex $y \in B_i$ and a path $Q' \subset H_i \setminus L$ connecting $v$ and $y$. Since the edge $e = uv$ does not belong to $L$, we can concatenate the path $Q$, the edge $e$, and the path $Q'$ to obtain a path in $H_i \setminus L$ connecting $x \in A_i$ and $y \in B_i$. However, this contradicts Claim~\ref{claim2}. Hence, $E_{H_i}(\overline{A}_i, \overline{B}_i) \subset L$, which implies that $E_{H_i}(\overline{A}_i, \overline{B}_i) = L$.   

     We now show that if the conjecture holds for infinite graphs, then it also holds for finite graphs. Indeed, let $G$ be a finite, bridgeless cubic graph. Consider $u,v \in V(G)$ with $u \neq v$ and $uv \in E(G)$. Let $H$ be the infinite, bridgeless cubic graph constructed from $G$ as in Lemma \ref{lemma3}. Fix $S$ to be the set of the two edges separating $G_0$ from the remainder of the graph $H$, where $G_0$ is a copy of $G$ with the edge $uv$ replaced by the edges in $S$.

     By Conjecture \ref{conjecture3} for infinite graphs, $H$ admits two perfect matchings $M_1$ and $M_2$ such that $M_1 \cap M_2$ contains no odd cut. By Lemma \ref{lemma3}, $M_1$ and $M_2$ induce perfect matchings $M_1'$ and $M_2'$ in $G$, respectively. Recall that either $S \subset M_i$ or $M_i \cap S = \emptyset$. Moreover, if $M_i$ contains $S$, then $M_i'$ contains the edge $uv$; if $M_i$ is disjoint from $S$, then $M_i'$ does not contain $uv$. We now show that $M_1' \cap M_2'$ contains no odd cut.

     Let $F' \subset E(G)$ be a cut such that $F' \subset M_1' \cap M_2'$, and let $A'$ and $B'$ denote the partition classes of the cut, so that $F' = E(A',B')$. If $uv \notin F'$, then $F' \subset M_1 \cap M_2$. We show that $F'$ is a cut of $H$. Indeed, if $uv \notin F'$, then $u$ and $v$ lie on the same side of the cut. Assume without loss of generality that $u,v \in A'$. We set $B = B'$ and define $A$ to be the union of $A'$ with the vertex set of the connected component of $H \setminus F'$ containing $u$ and $v$. It is immediate that $F' = E(A,B)$, and thus $F'$ is a cut of $H$. Since $F'$ is a cut of $H$ and $F' \subset M_1 \cap M_2$, it follows that $F'$ has even size. On the other hand, if $uv \in F'$, then $u$ and $v$ belong to different sides of the cut, say $u \in A'$ and $v \in B'$. Since $F' \subset M_1' \cap M_2'$, we have $uv \in M_1' \cap M_2'$, which implies $S \subset M_1 \cap M_2$. Let $f \in S$ be one of the edges in $S$. Then $F = (F' \setminus \{uv\}) \cup \{f\}$ is a cut of $H$. Indeed, suppose that $f$ is incident to $u$. We set $A = A'$ and define $B$ as the union of $B'$ with the connected component of $v$ in $H \setminus F$. It is immediate that $F = E(A,B)$ is a cut of $H$, as illustrated in Figure \ref{fig10}. Thus, since $F \subset M_1 \cap M_2$ and $F$ is a cut of $H$, $F$ has even size. Consequently, $F'$ also has even size, completing the proof.
     \begin{figure}[ht]
         \centering

\tikzset{every picture/.style={line width=0.75pt}} 

\begin{tikzpicture}[x=0.65pt,y=0.65pt,yscale=-1,xscale=1]

\draw [color={rgb, 255:red, 208; green, 2; blue, 27 }  ,draw opacity=1 ]   (108.5,193.5) .. controls (110.17,191.83) and (111.83,191.83) .. (113.5,193.5) .. controls (115.17,195.17) and (116.83,195.17) .. (118.5,193.5) .. controls (120.17,191.83) and (121.83,191.83) .. (123.5,193.5) .. controls (125.17,195.17) and (126.83,195.17) .. (128.5,193.5) .. controls (130.17,191.83) and (131.83,191.83) .. (133.5,193.5) .. controls (135.17,195.17) and (136.83,195.17) .. (138.5,193.5) .. controls (140.17,191.83) and (141.83,191.83) .. (143.5,193.5) -- (147,193.5) -- (147,193.5) ;
\draw [shift={(147,193.5)}, rotate = 0] [color={rgb, 255:red, 208; green, 2; blue, 27 }  ,draw opacity=1 ][fill={rgb, 255:red, 208; green, 2; blue, 27 }  ,fill opacity=1 ][line width=0.75]      (0, 0) circle [x radius= 3.35, y radius= 3.35]   ;
\draw [shift={(108.5,193.5)}, rotate = 0] [color={rgb, 255:red, 208; green, 2; blue, 27 }  ,draw opacity=1 ][fill={rgb, 255:red, 208; green, 2; blue, 27 }  ,fill opacity=1 ][line width=0.75]      (0, 0) circle [x radius= 3.35, y radius= 3.35]   ;
\draw   (147,98.8) .. controls (147,92.28) and (152.28,87) .. (158.8,87) -- (194.2,87) .. controls (200.72,87) and (206,92.28) .. (206,98.8) -- (206,197.2) .. controls (206,203.72) and (200.72,209) .. (194.2,209) -- (158.8,209) .. controls (152.28,209) and (147,203.72) .. (147,197.2) -- cycle ;
\draw [color={rgb, 255:red, 208; green, 2; blue, 27 }  ,draw opacity=1 ]   (108.5,161.5) .. controls (110.17,159.83) and (111.83,159.83) .. (113.5,161.5) .. controls (115.17,163.17) and (116.83,163.17) .. (118.5,161.5) .. controls (120.17,159.83) and (121.83,159.83) .. (123.5,161.5) .. controls (125.17,163.17) and (126.83,163.17) .. (128.5,161.5) .. controls (130.17,159.83) and (131.83,159.83) .. (133.5,161.5) .. controls (135.17,163.17) and (136.83,163.17) .. (138.5,161.5) .. controls (140.17,159.83) and (141.83,159.83) .. (143.5,161.5) -- (147,161.5) -- (147,161.5) ;
\draw [shift={(147,161.5)}, rotate = 0] [color={rgb, 255:red, 208; green, 2; blue, 27 }  ,draw opacity=1 ][fill={rgb, 255:red, 208; green, 2; blue, 27 }  ,fill opacity=1 ][line width=0.75]      (0, 0) circle [x radius= 3.35, y radius= 3.35]   ;
\draw [shift={(108.5,161.5)}, rotate = 0] [color={rgb, 255:red, 208; green, 2; blue, 27 }  ,draw opacity=1 ][fill={rgb, 255:red, 208; green, 2; blue, 27 }  ,fill opacity=1 ][line width=0.75]      (0, 0) circle [x radius= 3.35, y radius= 3.35]   ;
\draw [color={rgb, 255:red, 208; green, 2; blue, 27 }  ,draw opacity=1 ]   (108.5,129.5) .. controls (110.17,127.83) and (111.83,127.83) .. (113.5,129.5) .. controls (115.17,131.17) and (116.83,131.17) .. (118.5,129.5) .. controls (120.17,127.83) and (121.83,127.83) .. (123.5,129.5) .. controls (125.17,131.17) and (126.83,131.17) .. (128.5,129.5) .. controls (130.17,127.83) and (131.83,127.83) .. (133.5,129.5) .. controls (135.17,131.17) and (136.83,131.17) .. (138.5,129.5) .. controls (140.17,127.83) and (141.83,127.83) .. (143.5,129.5) -- (147,129.5) -- (147,129.5) ;
\draw [shift={(147,129.5)}, rotate = 0] [color={rgb, 255:red, 208; green, 2; blue, 27 }  ,draw opacity=1 ][fill={rgb, 255:red, 208; green, 2; blue, 27 }  ,fill opacity=1 ][line width=0.75]      (0, 0) circle [x radius= 3.35, y radius= 3.35]   ;
\draw [shift={(108.5,129.5)}, rotate = 0] [color={rgb, 255:red, 208; green, 2; blue, 27 }  ,draw opacity=1 ][fill={rgb, 255:red, 208; green, 2; blue, 27 }  ,fill opacity=1 ][line width=0.75]      (0, 0) circle [x radius= 3.35, y radius= 3.35]   ;
\draw [color={rgb, 255:red, 208; green, 2; blue, 27 }  ,draw opacity=1 ]   (108.5,99.5) .. controls (110.17,97.83) and (111.83,97.83) .. (113.5,99.5) .. controls (115.17,101.17) and (116.83,101.17) .. (118.5,99.5) .. controls (120.17,97.83) and (121.83,97.83) .. (123.5,99.5) .. controls (125.17,101.17) and (126.83,101.17) .. (128.5,99.5) .. controls (130.17,97.83) and (131.83,97.83) .. (133.5,99.5) .. controls (135.17,101.17) and (136.83,101.17) .. (138.5,99.5) .. controls (140.17,97.83) and (141.83,97.83) .. (143.5,99.5) -- (147,99.5) -- (147,99.5) ;
\draw [shift={(147,99.5)}, rotate = 0] [color={rgb, 255:red, 208; green, 2; blue, 27 }  ,draw opacity=1 ][fill={rgb, 255:red, 208; green, 2; blue, 27 }  ,fill opacity=1 ][line width=0.75]      (0, 0) circle [x radius= 3.35, y radius= 3.35]   ;
\draw [shift={(108.5,99.5)}, rotate = 0] [color={rgb, 255:red, 208; green, 2; blue, 27 }  ,draw opacity=1 ][fill={rgb, 255:red, 208; green, 2; blue, 27 }  ,fill opacity=1 ][line width=0.75]      (0, 0) circle [x radius= 3.35, y radius= 3.35]   ;
\draw   (49.5,95.1) .. controls (49.5,88.58) and (54.78,83.3) .. (61.3,83.3) -- (96.7,83.3) .. controls (103.22,83.3) and (108.5,88.58) .. (108.5,95.1) -- (108.5,193.5) .. controls (108.5,200.02) and (103.22,205.3) .. (96.7,205.3) -- (61.3,205.3) .. controls (54.78,205.3) and (49.5,200.02) .. (49.5,193.5) -- cycle ;
\draw    (444.5,148.5) ;
\draw [shift={(444.5,148.5)}, rotate = 0] [color={rgb, 255:red, 0; green, 0; blue, 0 }  ][fill={rgb, 255:red, 0; green, 0; blue, 0 }  ][line width=0.75]      (0, 0) circle [x radius= 3.35, y radius= 3.35]   ;
\draw [shift={(444.5,148.5)}, rotate = 0] [color={rgb, 255:red, 0; green, 0; blue, 0 }  ][fill={rgb, 255:red, 0; green, 0; blue, 0 }  ][line width=0.75]      (0, 0) circle [x radius= 3.35, y radius= 3.35]   ;
\draw   (483,53.8) .. controls (483,47.28) and (488.28,42) .. (494.8,42) -- (530.2,42) .. controls (536.72,42) and (542,47.28) .. (542,53.8) -- (542,152.2) .. controls (542,158.72) and (536.72,164) .. (530.2,164) -- (494.8,164) .. controls (488.28,164) and (483,158.72) .. (483,152.2) -- cycle ;
\draw [color={rgb, 255:red, 208; green, 2; blue, 27 }  ,draw opacity=1 ]   (444.5,116.5) .. controls (446.17,114.83) and (447.83,114.83) .. (449.5,116.5) .. controls (451.17,118.17) and (452.83,118.17) .. (454.5,116.5) .. controls (456.17,114.83) and (457.83,114.83) .. (459.5,116.5) .. controls (461.17,118.17) and (462.83,118.17) .. (464.5,116.5) .. controls (466.17,114.83) and (467.83,114.83) .. (469.5,116.5) .. controls (471.17,118.17) and (472.83,118.17) .. (474.5,116.5) .. controls (476.17,114.83) and (477.83,114.83) .. (479.5,116.5) -- (483,116.5) -- (483,116.5) ;
\draw [shift={(483,116.5)}, rotate = 0] [color={rgb, 255:red, 208; green, 2; blue, 27 }  ,draw opacity=1 ][fill={rgb, 255:red, 208; green, 2; blue, 27 }  ,fill opacity=1 ][line width=0.75]      (0, 0) circle [x radius= 3.35, y radius= 3.35]   ;
\draw [shift={(444.5,116.5)}, rotate = 0] [color={rgb, 255:red, 208; green, 2; blue, 27 }  ,draw opacity=1 ][fill={rgb, 255:red, 208; green, 2; blue, 27 }  ,fill opacity=1 ][line width=0.75]      (0, 0) circle [x radius= 3.35, y radius= 3.35]   ;
\draw [color={rgb, 255:red, 208; green, 2; blue, 27 }  ,draw opacity=1 ]   (444.5,85.5) .. controls (446.17,83.83) and (447.83,83.83) .. (449.5,85.5) .. controls (451.17,87.17) and (452.83,87.17) .. (454.5,85.5) .. controls (456.17,83.83) and (457.83,83.83) .. (459.5,85.5) .. controls (461.17,87.17) and (462.83,87.17) .. (464.5,85.5) .. controls (466.17,83.83) and (467.83,83.83) .. (469.5,85.5) .. controls (471.17,87.17) and (472.83,87.17) .. (474.5,85.5) .. controls (476.17,83.83) and (477.83,83.83) .. (479.5,85.5) -- (483,85.5) -- (483,85.5) ;
\draw [shift={(483,85.5)}, rotate = 0] [color={rgb, 255:red, 208; green, 2; blue, 27 }  ,draw opacity=1 ][fill={rgb, 255:red, 208; green, 2; blue, 27 }  ,fill opacity=1 ][line width=0.75]      (0, 0) circle [x radius= 3.35, y radius= 3.35]   ;
\draw [shift={(444.5,85.5)}, rotate = 0] [color={rgb, 255:red, 208; green, 2; blue, 27 }  ,draw opacity=1 ][fill={rgb, 255:red, 208; green, 2; blue, 27 }  ,fill opacity=1 ][line width=0.75]      (0, 0) circle [x radius= 3.35, y radius= 3.35]   ;
\draw [color={rgb, 255:red, 208; green, 2; blue, 27 }  ,draw opacity=1 ]   (444.5,54.5) .. controls (446.17,52.83) and (447.83,52.83) .. (449.5,54.5) .. controls (451.17,56.17) and (452.83,56.17) .. (454.5,54.5) .. controls (456.17,52.83) and (457.83,52.83) .. (459.5,54.5) .. controls (461.17,56.17) and (462.83,56.17) .. (464.5,54.5) .. controls (466.17,52.83) and (467.83,52.83) .. (469.5,54.5) .. controls (471.17,56.17) and (472.83,56.17) .. (474.5,54.5) .. controls (476.17,52.83) and (477.83,52.83) .. (479.5,54.5) -- (483,54.5) -- (483,54.5) ;
\draw [shift={(483,54.5)}, rotate = 0] [color={rgb, 255:red, 208; green, 2; blue, 27 }  ,draw opacity=1 ][fill={rgb, 255:red, 208; green, 2; blue, 27 }  ,fill opacity=1 ][line width=0.75]      (0, 0) circle [x radius= 3.35, y radius= 3.35]   ;
\draw [shift={(444.5,54.5)}, rotate = 0] [color={rgb, 255:red, 208; green, 2; blue, 27 }  ,draw opacity=1 ][fill={rgb, 255:red, 208; green, 2; blue, 27 }  ,fill opacity=1 ][line width=0.75]      (0, 0) circle [x radius= 3.35, y radius= 3.35]   ;
\draw   (385.5,50.1) .. controls (385.5,43.58) and (390.78,38.3) .. (397.3,38.3) -- (432.7,38.3) .. controls (439.22,38.3) and (444.5,43.58) .. (444.5,50.1) -- (444.5,148.5) .. controls (444.5,155.02) and (439.22,160.3) .. (432.7,160.3) -- (397.3,160.3) .. controls (390.78,160.3) and (385.5,155.02) .. (385.5,148.5) -- cycle ;
\draw    (483,148.5) ;
\draw [shift={(483,148.5)}, rotate = 0] [color={rgb, 255:red, 0; green, 0; blue, 0 }  ][fill={rgb, 255:red, 0; green, 0; blue, 0 }  ][line width=0.75]      (0, 0) circle [x radius= 3.35, y radius= 3.35]   ;
\draw [shift={(483,148.5)}, rotate = 0] [color={rgb, 255:red, 0; green, 0; blue, 0 }  ][fill={rgb, 255:red, 0; green, 0; blue, 0 }  ][line width=0.75]      (0, 0) circle [x radius= 3.35, y radius= 3.35]   ;
\draw [color={rgb, 255:red, 208; green, 2; blue, 27 }  ,draw opacity=1 ]   (444,192) .. controls (442.35,190.31) and (442.37,188.65) .. (444.06,187) .. controls (445.74,185.35) and (445.76,183.68) .. (444.11,182) .. controls (442.46,180.31) and (442.48,178.65) .. (444.17,177) .. controls (445.86,175.35) and (445.88,173.69) .. (444.23,172) .. controls (442.58,170.31) and (442.6,168.65) .. (444.29,167) .. controls (445.97,165.35) and (445.99,163.68) .. (444.34,162) .. controls (442.69,160.31) and (442.71,158.65) .. (444.4,157) .. controls (446.09,155.35) and (446.11,153.69) .. (444.46,152) -- (444.5,148.5) -- (444.5,148.5) ;
\draw [shift={(444.5,148.5)}, rotate = 0] [color={rgb, 255:red, 208; green, 2; blue, 27 }  ,draw opacity=1 ][fill={rgb, 255:red, 208; green, 2; blue, 27 }  ,fill opacity=1 ][line width=0.75]      (0, 0) circle [x radius= 3.35, y radius= 3.35]   ;
\draw [shift={(444,192)}, rotate = 270.66] [color={rgb, 255:red, 208; green, 2; blue, 27 }  ,draw opacity=1 ][fill={rgb, 255:red, 208; green, 2; blue, 27 }  ,fill opacity=1 ][line width=0.75]      (0, 0) circle [x radius= 3.35, y radius= 3.35]   ;
\draw    (482.5,192) -- (483,148.5) ;
\draw [shift={(483,148.5)}, rotate = 0] [color={rgb, 255:red, 0; green, 0; blue, 0 }  ][fill={rgb, 255:red, 0; green, 0; blue, 0 }  ][line width=0.75]      (0, 0) circle [x radius= 3.35, y radius= 3.35]   ;
\draw [shift={(482.5,192)}, rotate = 270.66] [color={rgb, 255:red, 0; green, 0; blue, 0 }  ][fill={rgb, 255:red, 0; green, 0; blue, 0 }  ][line width=0.75]      (0, 0) circle [x radius= 3.35, y radius= 3.35]   ;
\draw   (311,204.2) .. controls (311,197.46) and (316.46,192) .. (323.2,192) -- (625.8,192) .. controls (632.54,192) and (638,197.46) .. (638,204.2) -- (638,240.8) .. controls (638,247.54) and (632.54,253) .. (625.8,253) -- (323.2,253) .. controls (316.46,253) and (311,247.54) .. (311,240.8) -- cycle ;
\draw   (252,132.75) -- (274.2,132.75) -- (274.2,127) -- (289,138.5) -- (274.2,150) -- (274.2,144.25) -- (252,144.25) -- cycle ;

\draw (37,58.4) node [anchor=north west][inner sep=0.75pt]    {$G$};
\draw (110.5,196.9) node [anchor=north west][inner sep=0.75pt]    {$u$};
\draw (136,196.4) node [anchor=north west][inner sep=0.75pt]    {$v$};
\draw (70,141.4) node [anchor=north west][inner sep=0.75pt]    {$A$};
\draw (172,140.4) node [anchor=north west][inner sep=0.75pt]    {$B$};
\draw (346,13.4) node [anchor=north west][inner sep=0.75pt]    {$G_{0} =G-\{uv\}$};
\draw (446.5,151.9) node [anchor=north west][inner sep=0.75pt]    {$u$};
\draw (472,151.4) node [anchor=north west][inner sep=0.75pt]    {$v$};
\draw (406,96.4) node [anchor=north west][inner sep=0.75pt]    {$A$};
\draw (508,95.4) node [anchor=north west][inner sep=0.75pt]    {$B$};
\draw (435,218.4) node [anchor=north west][inner sep=0.75pt]    {$H-G_{0}$};
\draw (427,164.4) node [anchor=north west][inner sep=0.75pt]    {$f$};
\draw (120,73.4) node [anchor=north west][inner sep=0.75pt]  [color={rgb, 255:red, 208; green, 2; blue, 27 }  ,opacity=1 ]  {$F'$};
\draw (460,32.4) node [anchor=north west][inner sep=0.75pt]  [color={rgb, 255:red, 208; green, 2; blue, 27 }  ,opacity=1 ]  {$F$};

\end{tikzpicture}
         \caption{The cut $F'$ in $G$ induces a new cut $F$ in $H$ intersecting the $2$-cut $S$.
         }
         \label{fig10}
     \end{figure}
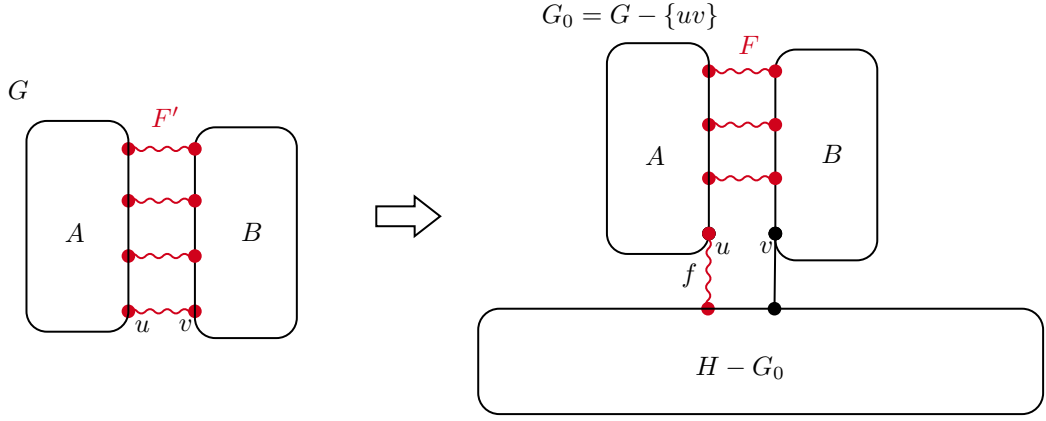
     \end{proof}

As an application, we obtain that the Fan-Raspaud conjecture for infinite graphs implies the conjecture of M\'{a}\v{c}ajov\'{a} and \v{S}koviera.

\begin{cor}\label{cor1}
    If the Fan-Raspaud conjecture holds for infinite graphs, then so does the conjecture of M\'{a}\v{c}ajov\'{a} and \v{S}koviera.
   
\end{cor}
\begin{proof}
    If Conjecture \ref{conjecture2} holds for infinite graphs, then by Theorem \ref{thm3}, Conjecture \ref{conjecture2} holds for finite graphs. Conjecture \ref{conjecture2} for finite graphs, in turn, implies Conjecture \ref{conjecture3} for finite graphs. By Theorem \ref{thm4}, Conjecture \ref{conjecture3} for finite graphs implies Conjecture \ref{conjecture3} for infinite graphs. Combining all these implications yields the desired result.
    
\end{proof}

To conclude the paper, we present as a corollary that Conjecture~\ref{conjecture4} holds for infinite graphs, just as it does for finite graphs, which was proved in~\cite{Disjointoddcircuitsinabridgeless}.

\begin{cor}\label{thm6}
     Every infinite, cubic, bridgeless graph $G$ has two perfect matching $M_1$, $M_2$ so that $G-(M_1\cup M_2)$ is bipartite.

\end{cor}
\begin{proof}
This follows from Lemmas~\ref{thm1} and \ref{lemma2}, combined with the fact that the $F$-limit of a sequence of bipartite graphs is also bipartite.

\end{proof}
\bibliography{referencias}
\bibliographystyle{plain}    
\Addresses 

\end{document}